\documentclass[journal,print]{ieeecolor} 

\usepackage{generic,xcolor,amsmath,amssymb,amsfonts,mathtools,graphicx,orcidlink,booktabs,balance,hyperref,cite,subcaption,float}
\newtheorem{thm}{Theorem}
\newtheorem{remark}{Remark}
\newtheorem{prop}{Proposition}

\newtheorem{problem}{Problem}
\newtheorem{lemma}{Lemma}

\newcommand{\T}{{T}}

\newcommand{\trace}{{\mathrm{trace}}}

\newcommand{\init}{{\rm init}}
\newcommand{\fin}{{\rm fin}}

\newcommand{\tr}{{\rm tr}}

\newsavebox\CBox

\hypersetup{
     colorlinks=true,
     linkcolor= blue,
     filecolor=black,      
     urlcolor=cyan,
     citecolor= blue,
     pdftitle={Overleaf Example},
    pdfpagemode=FullScreen,
     }

\title{\huge 
Minimizing Control Attention:\\
The Linear Gauss-Markov paradigm}

\author{Ralph Sabbagh, Asmaa Eldesoukey, Mahmoud Abdelgalil, and Tryphon T. Georgiou
\thanks{R. Sabbagh, A. Eldesoukey, and T. T. Georgiou are with the Department of Mechanical and Aerospace
Engineering, University of California, Irvine, CA 92697 USA; \{rsabbag1,aeldesou,tryphon\}@uci.edu.}
        \thanks{M. Abdelgalil is with the Department of Electrical and Computer Engineering, University of California, San Diego, CA 92093 USA; mabdelgalil@ucsd.edu.}
\thanks{Supported by AFOSR (FA9550-24-1-0278), ARO (W911NF-22-1-0292), and NSF (ECCS-2347357).}}

\begin{document}

\maketitle

\begin{abstract}
We revisit the concept of `attention' as a technical term to quantify the effort in calibrating control action based on available data. While Wiener, in his work on Cybernetics, anticipated key principles on prioritizing task-relevant signals, it was not until the late 1990's when Brockett first formulated pertinent optimization problems that have inspired subsequent as well as the present work.
`Attention,' as a technical term, is defined so as to quantify the dependence of the control law on the time and space/state coordinate; a control law that is independent of time and space, assuming it meets specifications, requires vanishing attention. In the present work we focus on Linear-Markovian dynamics with Gaussian state uncertainty so as to analyze the structure of {\em minimal-attention control} schemes that steer the dynamics between terminal states with Gaussian uncertainty profile.
\end{abstract}


\section{INTRODUCTION}
\IEEEPARstart{I}{nformation} and control have been inextricably linked since the infancy of the field when Wiener, in his famous work on `{\em Cybernetics or Control and Communication in the Animal and the Machine}' \cite{wiener2019cybernetics}, emphasized that ``the problems of control engineering and of communication engineering were
inseparable.'' Wiener went on to highlight the significance of abstracting the notion of a ``message'' -- a sequence of measurable events that are needed to carry out desired tasks.

Much of the subsequent work by Wiener and others focused on predicting, encoding, and filtering such messages, and later on, on optimizing control action on the basis of such messages/signals. Yet, the most basic question on how much attention is needed on the part of a regulating apparatus in carrying out a task, was not formulated, let alone addressed, until the work of Roger Brockett half a century after the appearance of `{\em Cybernetics},' almost to the date, in 1997 \cite{Brockett1997MinAttn}.

 Brockett's interest in `{\em minimum–attention control}' was motivated by control engineering practice, in an attempt to quantify implementation costs in abstract terms. 
 As a guiding principle, Brockett put forth the point of view that ``the easiest control law to implement is that of a constant input,'' and that ``anything else requires some attention'' \cite{Brockett1997MinAttn,Brockett2003MinimizingAttention}.
 
 From this vantage point, he argued that for the case where the control action
 $U(t,X)$ is a function of time $t$ and the state $X$ of the dynamical system at hand, the dependence of $U$ on these two variables, namely, $\partial_tU$ and $\nabla_X U$, must be explicitly acknowledged and suitably penalized. Indeed, even without taking into account processing costs, interrogation that seeks the current time and the current value of the state, is tantamount to the attention that a controller must pay in a colloquial sense. Thus, Brockett argued \cite{Brockett1997MinAttn}, the variability of the control law $U(t,X)$
  with respect to changes in both the state $X$ and time $t$ is intrinsically linked to resources that are needed to physically implement the law. To this end, for the special case where our task is to steer the first-order system
  \[
  \dot X_t=U(t,X_t),
  \]
Brockett proposed the functional
\begin{align}\label{eq:etaa}
   \eta_a=\frac{1}{2}\int_0^\tau\int_{\mathbb{R}^n} \|\nabla_{X}U\|^2+\|\partial_tU\|^2\,\text{d}X\text{d}t,
\end{align}
 as a suitable measure for the implementation cost associated with a particular feedback control Law $U$ \cite{Brockett1997MinAttn,Brockett2003MinimizingAttention,Brockett2012Liouville}. His original insight has since permeated resource-aware control \cite{Jang2015BallCatching,Anta2010MinAttnAnytime,Donkers2014MinAttnLinear,heemels2012introduction,Dirr2016EnsembleMeanVariance,NagaharaNesic2020CDC,nowzari2016distributed,eldesoukey2025collective}, where decisions about when to measure or update are seen as an integral part of the control design, and where sparsity of interventions is a feature rather than an artifact.

Building on Brockett’s original vision, we study covariance steering for linear {\color{black}finite-dimensional} stochastic systems in which attention is explicitly priced. Specifically, we address the problem of steering the covariance $\Sigma_t:=\mathbb{E}[X_tX_t^\top]$ of the (zero-mean) Gauss-Markov process 
\begin{align}\label{eq:X}
    \text{d}X_t = A_tX_t\ \text{d}t + B\, \text{d}W_t,\;\;\;X_0\sim \mathcal N(0,\Sigma_0),
\end{align}
where $W_t$ is a Wiener process, via a control $U(t,X_t)=A_tX_t$ that
 minimizes a probabilistic version of Brockett's attention functional\footnote{For linear dynamics and an unbounded state space, \eqref{eq:etaa} need not be bounded. A natural adjustment is to consider the expectation of the integrand--focusing our attention to states that are more likely to be reached.}. This functional takes the form
 \begin{align}\label{eq:J}
 \begin{split}
 J_\alpha(A,\Sigma):= &  \, \alpha \int_0^\T \tr(A_tA_t^\top)\, \text{d}t \\
  & +(1-\alpha)\int_0^\T \tr(\dot{A}_t\Sigma_t\dot{A}_t^\top)\, \text{d}t.
  \end{split}
\end{align}
The parameter $\alpha\in[0,1]$ weighs in the relative importance of the two terms, referred to as the {\em spatial} and {\em temporal} attention, respectively. The spatial attention 
$$
\int_0^T\mathbb E\left[\|\nabla_X U\|^2_F\right]{\rm d}t =\int_0^T\tr(A_tA_t^\top){\rm d}t,
$$
penalizes how fast the feedback varies across space, whereas the temporal attention 
$$
\int_0^T\mathbb E\left[\|\partial_tU\|^2_F\right]{\rm d}t=\int_0^T\tr(\dot{A}_t\Sigma_t\dot{A}_t^\top){\rm d}t,
$$
penalizes how rapidly the feedback gain changes with time. Throughout, $\|\cdot\|_F$ denotes the Frobenius norm.

Whereas the above terms are reminiscent of quadratic cost functionals in linear quadratic regulator theory, covariance steering, and optimal transport \cite{Astrom,chen2021stochastic,ChenGeorgiouPavon2016_TAC2,Villani2009OptimalTransport}, the problem of minimizing \eqref{eq:J} subject to the dynamical and boundary constraints considered herein is decidedly different. For one thing, the Gauss-Markov model contains a product term of control and state, and hence the dynamics are bilinear. Moreover, the departure from earlier paved paths in covariance steering is much deeper.  There is no apparent re-parametrization that renders the problem convex\footnote{Functionals of the form \label{footnote2} $\int\tr(A_t\Sigma_tA_t){\rm d}t$ in covariance steering and optimal transport turn out to be convex when re-parametrized in $(U,A)$ with $U:=\Sigma A$ being the control and $\Sigma$ being the dynamical variable; an analogous change of parameters that may render \eqref{eq:J} convex does not seem likely.}. This underscores the need for a substantially more involved mathematical analysis, to establish the existence and the nature of solutions, as pursued herein.

In this work we formulate the problem of minimum attention steering in a Gauss-Markov setting. The main results and new insights gained are as follows:
\begin{itemize}
  \item[-] \emph{Existence of solutions.} We establish existence of solutions via the direct method in the calculus of variations.
  %
  \item[-] \emph{First–order optimality conditions and regularity upgrade.}
  We provide optimality conditions and exploit the structure to upgrade the regularity of solutions.
  %
  \item[-] \emph{Minimal attention for $\alpha\in\{0,1\}$.} We analyze separately the limiting cases when $\alpha$ is either $1$ (spatial attention) or $0$ (temporal attention).
  \item[-] \emph{Information–geometric links.} We compare `small-attention' trajectories to Fisher–Rao geodesics. 
\end{itemize}
In sum, we establish existence, optimality conditions, sharp regularity, closed–form special cases, and an information–geometric comparison principle for attention–penalized covariance steering beyond specialized cases.

\subsubsection*{Organization of this article} We include notation and some preliminaries in Section \ref{sec:notation}, and then proceed with the problem formulation in Section \ref{sec:main-problem}. Section \ref{sec:Existence}, contains the main results of the paper on the existence and regularity of minimum attention control schemes and of the corresponding trajectories. Results on the existence and structure of the optimal control schemes for the two special and important cases, where $\alpha$ is either $1$ or $0$, are presented in Sections \ref{sec:spatial} and \ref{sec:temporal}, respectively. An information-geometric link between the attention functional and the Fisher-Rao action is discussed in Section \ref{sec:fisher}. We close with an academic example in Section \ref{sec:examples}, showcasing the behavior of numerically-attained minimizers, and concluding remarks in Section \ref{sec:conclusions}.

\section{Notation and preliminaries}\label{sec:notation}
Let $\mathbb{R}^{n\times m}$ denote the set of $n$ by $m$ matrices, and $\mathbb{S}^n$, $\mathbb{S}_{+}^n$, $\mathbb{A}^n\subset \mathbb{R}^{n\times n}$ the symmetric, positive definite,  and antisymmetric matrices, respectively. Let $\mathcal{H}$ be the Hilbert space formed by the direct sum:
\begin{align*}
    \mathcal{H}=H^1([0,\T];\mathbb{R}^{n\times n})\oplus H^1([0,\T];\mathbb{S}^n),
\end{align*}
where $H^1([0,\T];\mathbb{R}^{n\times n})$ and $H^1([0,\T];\mathbb{S}^n)$ are the Sobolev spaces of $\mathbb{R}^{n\times n}$ and $\mathbb{S}^n$-functions that are in $L^2([0,\T])$ along with their weak first-derivatives. Recall that $H^1([0,\T]; \cdot):=W^{1,2}([0,\T]; \cdot)$ can be identified with the space of absolutely continuous functions with square-integrable derivatives; more generally, the same is true for $W^{1,p}([0,\T];\cdot),~p\geq 1$ with derivatives in $L^p([0,\T];\cdot)$ \cite[Ch.\ 5]{evans2022partial}. 
Throughout this work we view elements in Sobolev spaces, interchangeably, as either equivalence classes of functions or continuous representatives; thus, we write the $H^1$-inner product
\[
\langle Y,Z\rangle_{H^1}:=\int_0^\T \left(\tr(Y_tZ_t^\top)+ \tr(\dot Y_t\dot Z_t^\top)\right)\, \text{d}t.
\]
The space  $H_0^1([0,T])\subseteq H^1([0,T])$ consists of functions vanishing at the endpoints $0$ and $T$, and $H^{-1}$ denotes the continuous linear functionals on $H_0^1$. These form the {\em Gelfand triple}: $ H_0^1\subset L^2 \subset H^{-1}$ \cite{GelfandVilenkin1964}.

We denote by $\|\cdot\|$ and $\|\cdot\|_F$ the operator and Frobenius norms of a matrix, respectively, then, $\|\cdot\|\leq \|\cdot\|_F\leq \sqrt{n}\|\cdot\|$. We use a similar notation $\|\cdot\|_X$ to denote the norm of a Banach space {$X$}, typically for $X\in\{L^1,L^2,H^1,C^0\}$. We recall that the norm-squared in a Hilbert space is weakly lower semi-continuous, i.e., for any weakly converging sequence $x^k\rightharpoonup x$, $ \liminf_{k\rightarrow\infty}\|x_k\|^2\geq \|x\|^2$.
Additionally, we note that the matrix square root function is Lipschitz continuous on a set of positive-definite matrices that are bounded below by a positive constant \cite[p.\ 304]{bhatia2013matrix}.

We finally recall that a map $F:X\to Y$ is \emph{Fr\'echet differentiable at $x\in X$} if there exists a bounded linear map
$L$ from $X$ to $Y$ (denoted $L\in \mathcal{B}(X,Y)$) such that
\begin{align}
\lim_{\|h\|_X\to 0}\frac{\|\,F(x+h)-F(x)-L(h)\,\|_Y}{\|h\|_X}=0.\label{def:frechet}
\end{align}
The map $L$ (unique if it exists) is called the \emph{Fr\'echet derivative} of $F$ at $x$ and is denoted by
$D_xF(x)$. Equivalently,
\[
F(x+h)=F(x)+D_xF(x)(h)+o(\|h\|_X)\qquad(h\to 0).
\]
If $F$ is Fr\'echet differentiable at every point of $X$ and $x\mapsto D_xF(x)$ is continuous as a map
$X\to\mathcal B(X,Y)$, then $F$ is said to be of class $C^1$ from $X$ to $Y$.

\section{Problem formulation}\label{sec:main-problem}

Consider the $n$-dimensional It$\hat{\text{o}}$ process $X_t$, satisfying
\begin{align}\label{eq:Ito}
    \, \text{d}X_t = A_tX_t\, \text{d}t + B\, \text{d}W_t,
\end{align}
for $t \in [0,T]$ and $X_0$ an $n$-dimensional centered normal random variable. Throughout, $A_t\in\mathbb{R}^{n\times n}$ is a control matrix,  $B$ $\in\mathbb{R}^{n\times m}$, and $W_t$ is a standard $m$-dimensional Wiener process\footnote{Although  $B$ is taken to be constant for simplicity, the results extend to $B_t$ being a function of time with sufficient regularity.}.
One can readily verify that the state covariance $\Sigma_t:=\mathbb{E}_t[X_tX_t^\top]$ obeys the differential Lyapunov equation
      \begin{align}
\dot{\Sigma}_t&=A_t\Sigma_t+\Sigma_t A_t^\top+BB^\top,\label{eq:constr1}
\end{align}
for $t\in[0,\T]$.
The problem addressed in this paper is to steer the state covariance between endpoint specifications while minimizing the attention functional \eqref{eq:J}, as stated next.
\begin{problem} \label{problem:main-min-attention}
 Given $\T>0$, $\alpha\in(0,1)$, $B\in\mathbb{R}^{n\times m}$, and $\Sigma_\init$, $\Sigma_\fin\in\mathbb{S}_{+}^n$, determine
  \begin{align*}
    \min_{(A,\Sigma)\in\mathcal{H}} J_\alpha(A,\Sigma),
  \end{align*}
  subject to \eqref{eq:constr1}, and the end-point specifications
 \begin{align}
\Sigma_0&=\Sigma_\init,~\Sigma_\T=\Sigma_\fin.\label{eq:constr2} 
  \end{align}
  
\end{problem}

In the next section we establish existence of solutions and study their properties.

\section{Main results}\label{sec:Existence}

We first establish existence of solutions to Problem~\ref{problem:main-min-attention} using the direct method in the calculus of variations, and prove that optimal paths \{$\Sigma_t^\star,~t\in[0,\T]$\} are uniformly bounded in the interior of $\mathbb{S}_{+}^n$. Subsequently, we derive the first-order necessary conditions and show that minimizers are smooth in time.  Note that since Problem \ref{problem:main-min-attention} is not convex in $(A,\Sigma)$, it is not clear that uniqueness of a minimizer can be guaranteed at the outset (see footnote \ref{footnote2}).

\subsection{Existence of a minimizer and uniform bounds}

\begin{thm}
\label{eq:thm}
Problem~\ref{problem:main-min-attention} admits a solution $(A^\star,\Sigma^\star)\in\mathcal{H}$. Moreover, any  solution $\Sigma^\star$ satisfies
\begin{align}\label{eq:bounds}
   cI\leq  \Sigma_t^\star \leq CI,~~\forall t\in[0,\T],
\end{align}
where 
\begin{align*}
    c &:= e^{-2\sqrt{\T J^\star_\alpha/\alpha}} \lambda_{\min}(\Sigma_{\init}),~\text{and}\\
  C&:= e^{2\sqrt{\T J^\star_\alpha/\alpha}} (\lambda_{\max}(\Sigma_{\init})+\T\lambda_{\max}(BB^\top)).  
\end{align*}
\end{thm}

\begin{proof}
We first establish feasibility. Consider the path 
    \begin{align*}
        \Sigma_t=(1-t/\T)\Sigma_\init+(t/\T)\Sigma_\fin,~t\in[0,\T],
    \end{align*}
    induced by the control 
    \begin{align*}
        A_t=\frac{1}{2}(\dot{\Sigma}_t-BB^\top)\Sigma_t^{-1}.
    \end{align*}
The pair $(A,\Sigma)$ is feasible, since  $A$,  $\Sigma$, and
    \begin{align*}
        \dot{\Sigma}&=(\Sigma_\fin-\Sigma_\init)/\T,\\
        \dot{A}&=-(\dot{\Sigma}-BB^\top)\Sigma^{-1}\dot{\Sigma}\Sigma^{-1}/2,
    \end{align*}
    are continuous functions of time. Thus, the infimum of $J_\alpha\geq 0$ over the feasible set, denoted by $\inf J_\alpha$, exists and satisfies
    \begin{align*}
        0\leq \inf J_\alpha \leq J_\alpha(A,\Sigma) <\infty. 
    \end{align*}
 
By definition of the infimum, there exists a sequence $(A^k,\Sigma^k)\in\mathcal{H}$, $k\in \mathbb{N}$, where $\mathbb{N}$ denotes the natural numbers (excluding $0$), satisfying (\ref{eq:constr1}-\ref{eq:constr2}) with
  \begin{align}\nonumber
\lim_{k\rightarrow\infty}J_\alpha(A^k,\Sigma^k)&=\inf J_\alpha,\\\label{eq:M}
      J_\alpha(A^k,\Sigma^k)&\leq \inf J_\alpha +\frac{1}{k}=:M_k<\infty.
  \end{align}
  Using the Cauchy-Schwarz inequality, $\|A^k\|_{L^1}\leq \sqrt{\T}  \|A^k\|_{L^2}$, and combining with \eqref{eq:M}, we have that
  \begin{align}\label{eq:inequalities}
    \|A^k\|_{L^1}\leq \sqrt{\T}  \|A^k\|_{L^2}\leq \sqrt{\frac{\T M_k}{\alpha}}.
  \end{align}
 Now, let $\Phi^k(t,s)$ be the unique solution of the matrix ODE $\partial_t{\Phi}^k(t,s)=A^k\Phi^k(t,s)$, with $\Phi^k(s,s)=I$. This is absolutely continuous and satisfies the standard matrix inequalities 
 \begin{align*}
     \|\Phi^k(t,s)\|&\leq 1+\int_s^t\|A^k_u\|\|\Phi^k(u,s)\|\, \text{d}u,\\
     \|\Phi^k(t,s)^{-1}\|&\leq 1+\int_s^t\|A^k_u\|\|\Phi^k(u,s)^{-1}\|\, \text{d}u,
 \end{align*} 
for $s,t\in[0,\T]$. By Gr\"onwall's inequality, the maximum and minimum singular values of $\Phi^k(t,s)$ are bounded, as
  \begin{align*}
      \sigma_{\max}(\Phi^k(t,s))\!:=& \|\Phi^k(t,s)\|\leq e^{\int_s^t\|A^k_u\|\, \text{d}u}\leq e^{\sqrt{\frac{\T M_k}{\alpha}}},\\
       \sigma_{\min}(\Phi^k(t,s))\!:=&\|\Phi^k(t,s)^{-1}\|^{-1}\geq e^{-\int_s^t\|A^k_u\|\, \text{d}u}\geq e^{-\sqrt{\frac{\T M_k}{\alpha}}}\!\!,
  \end{align*}
  where the rightmost inequalities follow from $\|\cdot\|\leq \|\cdot\|_F$, and then \eqref{eq:inequalities}. Note also that $M_k<M_1$ giving bounds independent of $k$.
  It follows that the minimal and maximal eigenvalues of 
  \begin{align*}
\Sigma_t^k=\Phi^k(t,0)\Sigma_\init\Phi^k(t,0)^\top+\int_0^t\Phi^k(t,s)BB^\top\Phi^k(t,s)^\top\, \text{d}s,
  \end{align*}
  are bounded (via their Rayleigh quotients) by
  \begin{align*}
      \lambda_{\min}(\Sigma_t^k)&\geq\sigma_{\min}(\Phi^k(t,0))^2 \lambda_{\min}(\Sigma_\init)\\
      &\geq e^{-2\sqrt{\frac{\T M_k}{\alpha}}}\lambda_{\min}(\Sigma_\init)=:c_k>0, \mbox{ and}\\
      \lambda_{\max}(\Sigma_t^k)&\leq \sigma_{\max}(\Phi^k(t,0))^2 \lambda_{\max}(\Sigma_\init)\\
      &\hspace*{-12pt}+\sigma_{\max}(B)^2\int_0^t\sigma_{\max}(\Phi^k(t,s))^2\, \text{d}s\\
      &\hspace*{-12pt}\leq e^{2\sqrt{\frac{\T M_k}{\alpha}}}(\lambda_{\max}(\Sigma_\init)+\T \sigma_{\max}(B)^2)=:C_k <\infty.
  \end{align*}
Thus, we obtain the uniform bounds 
 \begin{align}
     c_1 I< c_k I \leq \Sigma_t^k \leq C_k I < C_1 I,\label{eq:sbounds}
 \end{align}
 for $t\in[0,\T]$ and $k\in\mathbb{N}$. 
It follows, 
\begin{align*}
    J_\alpha(A^k,\Sigma^k)\geq \alpha\int_0^\T\|A_t^k\|_F^2\, \text{d}t+(1-\alpha)c_1\int_0^\T\|\dot{A}_t^k\|_F^2\, \text{d}t.
\end{align*}
Hence, $(A^k,\Sigma^k)$ is uniformly bounded in $\mathcal{H}$ since 
\begin{align}
\|A^k\|^2_{H^1}&=\|A^k\|^2_{L^2}+\|\dot{A}^k\|^2_{L^2}\leq \frac{M_1}{\min(\alpha,(1-\alpha)c_1)},\\\nonumber
    \|\Sigma^k\|^2_{H^1}&=\|\Sigma^k\|^2_{L^2}+\|\dot{\Sigma}^k\|^2_{L^2}\\\nonumber
    &\leq n\T C_k^2+(2\sqrt{n}C_k\|A^k\|_{L^2}+\sqrt{\T}\|BB^\top\|_F)^2\\\nonumber
    &\leq  n\T C_1^2+(2\sqrt{n}C_1\sqrt{\frac{M_1}{\alpha}}+\sqrt{\T}\|BB^\top\|_F)^2,
\end{align}
using \eqref{eq:constr1} and the sub-multiplicative property of the Frobenius norm to bound $\|\dot{\Sigma}^k\|^2_{L^2}$. Since $\mathcal H$ is reflexive and $(A^k,\Sigma^k)$ is bounded in $\mathcal H$, the {\em Banach-Alaoglu theorem} \cite[Ch.\ 3]{Rudin1991FA} implies that there exists a subsequence (not relabeled) and a limit $(A^\star,\Sigma^\star)\in\mathcal H$
such that $(A^k,\Sigma^k)\rightharpoonup(A^\star,\Sigma^\star)$ in $\mathcal H$.

We now appeal to the {\em Rellich–Kondrachov compactness theorem} \cite{evans2022partial}, which states that
 $H^1([0,\T])$ is compactly embedded in $C^0([0,\T])$. Thus, we pass to another subsequence (again keeping the index $k$ for simplicity), so that $(A^k,\Sigma^k)$ converges also to $(A^\star,\Sigma^\star)$  uniformly (uniform convergence implies strong convergence in $L^2$; since the same subsequence
converges weakly in $H^1$ and $H^1\hookrightarrow L^2$ continuously, the $L^2$–limit
is unique and coincides with the weak $H^1$–limit).
 To recap, $(A^\star,\Sigma^\star)\in\mathcal{H}$ is such that 
\begin{align*}
    &(A^k,\Sigma^k)\rightharpoonup (A^\star,\Sigma^\star)\text{ in } H^1([0,\T];\mathbb{R}^{n\times n}\times\mathbb{S}^{n}),\\
    &(A^k,\Sigma^k)\rightarrow(A^\star,\Sigma^\star)\text{ in } 
     C^0([0,\T];\mathbb{R}^{n\times n}\times\mathbb{S}^{n}).
\end{align*}

Next, we show that the limit is feasible.
The uniform convergence implies that $\Sigma^\star$ satisfies the endpoint constraints \eqref{eq:constr2}. Moreover, considering
    \begin{align*}
\dot{\Sigma}^k&=A^k\Sigma^k+\Sigma^k{(A^k)}^\top+BB^\top~\text{a.e. on }[0,\T],
\end{align*}
the right hand side converges to $A_t^\star\Sigma_t^\star+\Sigma_t^\star{(A_t^\star)}^\top+BB^\top$ uniformly\footnote{Since $A^k\to A^\star$ and $\Sigma^k\to\Sigma^\star$ in $C^0$ and are uniformly bounded,
we have $A^k\Sigma^k+\Sigma^k A^{k\top}\to A^\star\Sigma^\star+\Sigma^\star A^{\star\top}$ in $C^0$.}, while the left hand side converges weakly in $L^2$ (since $\Sigma^k\rightharpoonup \Sigma^\star$ in $H^1$). By the uniqueness of the limit in $L^2$, \eqref{eq:constr1} holds a.e.\ in $[0,\T]$ for the limit. Thus, $(A^\star,\Sigma^\star)$ is feasible.

Finally, we show that $(A^\star,\Sigma^\star)$ minimizes \eqref{eq:J}. To this end, we show that $J_\alpha$ satisfies
 \begin{align}\label{eq:semicontinuity}
\liminf_{k\rightarrow\infty} J_\alpha(A^k,\Sigma^k)\geq J_\alpha(A^\star,\Sigma^\star),
 \end{align}
from which the desired statement $J_\alpha(A^\star,\Sigma^\star)=\inf J_\alpha$ readily follows.
Recall that the norm-squared in any Hilbert space is weakly lower semicontinuous, thus it is sufficient to show that both $A^k\rightharpoonup A^\star$ and $\dot{A}_t^k(\Sigma_t^k)^{1/2}\rightharpoonup \dot{A}^\star_t(\Sigma^\star_t)^{1/2}$ in $L^2([0,\T])$. The former is immediate since we know that $A^k\rightharpoonup A^\star\text{ in } H^1([0,\T])\subset L^2([0,\T])$. For the latter, we first observe that
for any $ Z\in L^2([0,\T])$,
\begin{align*}
     |\langle Z,\dot{A}^k(\Sigma^k)^{1/2}&-\dot{A}^\star(\Sigma^\star)^{1/2}\rangle_{L^2}|\\
       & \leq |\langle Z,\dot{A}^k((\Sigma^k)^{1/2}-(\Sigma^\star)^{1/2})\rangle_{L^2}|\\
     & \hspace{10pt} +|\langle Z,(\dot{A}^k-\dot{A}^\star)(\Sigma^\star)^{1/2}\rangle_{L^2}|\\
     &\leq \|Z\|_{L^2}\|\dot{A}^k\|_{L^2}\|(\Sigma^k)^{1/2}-(\Sigma^\star)^{1/2}\|_{C^0}\\
     &\hspace*{10pt}+|\langle Z(\Sigma^\star)^{1/2},(\dot{A}^k-\dot{A}^\star)\rangle_{L^2}|.
\end{align*}

The right-hand side tends to $0$ by the Lipschitz continuity of the matrix square root and the fact that $\dot{A}^k\rightharpoonup\dot{A}^\star$ in $L^2$. Note that 
$A^k$ is uniformly bounded in $H^1\subset L^2$, 
$Z(\Sigma^\star)^{1/2}\in L^2$, $\Sigma^k\rightarrow\Sigma^\star\text{ in }C^0$, while
$\Sigma^\star$ is uniformly bounded below, as we detail next by showing that the two-sided spectral bounds \eqref{eq:sbounds} hold for $\Sigma^\star$. Using Weyl's inequalities \cite[p.\ 63]{bhatia2013matrix},
\begin{align*}
    |\lambda_{\max}(\Sigma_t^k)-\lambda_{\max}(\Sigma_t^\star)|\leq \|\Sigma_t^k-\Sigma_t^\star\|,\\
    |\lambda_{\min}(\Sigma_t^k)-\lambda_{\min}(\Sigma_t^\star)|\leq \|\Sigma_t^k-\Sigma_t^\star\|.
\end{align*}
Thus, 
\begin{align*}
    \lambda_{\max}(\Sigma_t^\star)&\leq \lambda_{\max}(\Sigma_t^k)+\|\Sigma_t^k-\Sigma_t^\star\|\\
    &\leq C_k+\|\Sigma^k-\Sigma^\star\|_{C^0},\\
     \lambda_{\min}(\Sigma_t^\star)&\geq \lambda_{\min}(\Sigma_t^k)-\|\Sigma_t^k-\Sigma_t^\star\|\\
    &\geq c_k-\|\Sigma^k-\Sigma^\star\|_{C^0}.
\end{align*}
By the uniform convergence of $\Sigma^k$, the spectral bounds pass to the limit, so that 
     $c_1 I< c I \leq \Sigma_t^\star \leq C I < C_1 I$, 
 for $t\in[0,\T]$ and $k\in\mathbb{N}$, with 
 $c:=\lim_{k\rightarrow\infty}c_k$ and $C:=\lim_{k\rightarrow\infty}C_k$ as claimed in the theorem. This completes the proof.
\end{proof}

\begin{remark}[Relaxed bounds]
Utilizing any admissible value for the attention functional $J_\alpha$ in the uniform bounds of optimal trajectories $\Sigma^\star$, instead of the optimal value as in the theorem, yields conservative bounds that can be computed.
\end{remark}
%

\subsection{First-order necessary conditions (FONC)}

We recall first the Lagrange multiplier theorem for Hilbert spaces\footnote{The theorem applies more generally, to Banach spaces and local minimizers within a feasible neighborhood.} \cite{BonnansShapiro2000,Troeltzsch2010}, and then
derive first-order necessary conditions that must be satisfied by the minimizer of Problem~\ref{problem:main-min-attention}. 

\begin{thm}[Lagrange multiplier in Hilbert spaces]\label{thm:hilbert-kkt}
Let $X,Y$ be real Hilbert spaces and let
\[
J:X\to\mathbb R,\qquad \mathcal C:X\to Y,
\]
be Fr\'echet $C^1$ maps. Suppose $x^\star\in X$ satisfies $\mathcal C(x^\star)=0$, $J(x^\star)\leq J(x),~\forall x\in X~\text{ with }~\mathcal{C}(x)=0$, and suppose that the Fr\'echet derivative of $\mathcal{C}$ at $x^\star$, $D_{x}\mathcal C(x^\star):X\to Y$, is surjective. Then, there exists a multiplier $\lambda^\star\in Y$ such that $D_x\mathcal L(x^\star,\lambda^\star)=0$ where
$\mathcal L(x,\lambda):=J(x)+\langle \lambda,\mathcal C(x)\rangle_Y$.
\end{thm}

\begin{thm}[First-order necessary conditions]\label{thm:KKT-aug}
Let $(A^\star,\Sigma^\star)\in\mathcal{H}$ be a  minimizer of Problem~\ref{problem:main-min-attention}. Then there exists multipliers  $\Lambda\in W^{1,1}([0,T];\mathbb{S}^n)\subset  C^0([0,T];\mathbb{S}^n)$, $M_0,M_T\in\mathbb{S}^n$, such that 
\begin{subequations}\label{eq:neccessary-Eqs}
\begin{align}
\dot\Sigma^\star&=A^\star\Sigma^\star+\Sigma^\star A^{\star\top}+BB^\top,\label{eq:primal}\\
-\dot\Lambda\;\;&=\Lambda A^\star+A^{\star\top}\Lambda-(1-\alpha)\dot A^{\star\top}\dot A^\star,\label{eq:adjoint}\\
\Lambda\Sigma^\star&=\alpha\,A^\star-(1-\alpha)\tfrac{d}{dt}(\dot A^\star\Sigma^\star),\label{eq:stationarity}\\
\dot A_0^\star&=0,\qquad \dot A_T^\star=0,\label{eq:BC-A}\\
\Lambda_0&=M_0,\phantom{xx} \Lambda_T=-M_T,\label{eq:trans}
\end{align}
\end{subequations}
where equations \eqref{eq:primal}--\eqref{eq:stationarity} hold a.e. on $[0,T]$ and  $\dot{A}^\star\Sigma^\star\in W^{1,1}([0,T];\mathbb{R}^{n\times n}) \subset C^0([0,T];\mathbb{R}^{n\times n})$.
\end{thm}
\begin{proof}
    Consider the Hilbert spaces
    \begin{align*}
        X:=\mathcal{H},\qquad Y:=L^2([0,T];\mathbb{S}^n)\oplus\mathbb{S}^n\oplus\mathbb{S}^n,
    \end{align*}
    the functional $J_\alpha:X\rightarrow\mathbb{R}$  in \eqref{eq:J} and the map $\mathcal C:X\to Y$, 
 \[
\mathcal C(A,\Sigma):=(\dot\Sigma-A\Sigma-\Sigma A^\top-BB^\top,\Sigma_0-\Sigma_\init, 
\Sigma_T-\Sigma_\fin).
\]
We verify that the Fr\'echet derivatives of $J_\alpha$ and $\mathcal{C}$ given by 
\begin{align*}
      D_xJ_\alpha(A,\Sigma)[H,K]&=2\alpha\langle A,H\rangle_{L^2}+2(1-\alpha)\langle \dot{A}\Sigma,\dot{H}\rangle_{L^2}\\
      &+(1-\alpha)\langle \dot{A}^\top\dot{A},K\rangle_{L^2},
\end{align*}
    \[
    D_x\mathcal{C}(A,\Sigma)[H,K]=(\dot{K}-AK-KA^\top-H\Sigma-\Sigma H^\top,K_0,K_T),
\]
 are $C^1$ maps\footnote{One can check that, at any $(A,\Sigma)\in X$, the maps $D_xJ_\alpha(A,\Sigma):X\rightarrow \mathbb{R}$ and $D_x\mathcal{C}(A,\Sigma):X\rightarrow Y$ are well-defined, linear, bounded, and satisfy \eqref{def:frechet}.}.
To this end, consider a pair of points $(A,\Sigma)$ and $(A^{'},\Sigma^{'})$ in a neighborhood $U\subset X$ and any unit direction $(H,K)$ in $X$. Then,
\begin{align*}
    |(D_xJ_\alpha(A,\Sigma)-D_xJ_\alpha(A^{'},\Sigma^{'}))[H,K]|\leq T_1+T_2+T_3,
\end{align*}
where (using the continuous embedding $H^1(0,T)\hookrightarrow C^0([0,T])$,
with $\|f\|_{C^0}\le C_T\|f\|_{H^1}$ and the {\em Cauchy-Schwarz inequality}), 
\begin{align*}
    T_1 &= 2\alpha|\langle A-A^{'},H\rangle_{L^2}|\leq c_1\|A-A^{'}\|_{H^1}.\\
    T_2&=2(1-\alpha)|\langle \dot{A}\Sigma-\dot{A}^{'}\Sigma^{'},\dot{H}\rangle_{L^2}|\\
    &\leq 2(1-\alpha)(\|(\dot{A}-\dot{A}^{'})\Sigma\|_{L^2}+\|\dot{A}^{'}(\Sigma-\Sigma^{'})\|_{L^2})\\
    &\leq 2(1-\alpha)(\|A-A^{'}\|_{H^1}\|\Sigma\|_{C^0}+\|\Sigma-\Sigma^{'}\|_{C^0}\|A^{'}\|_{H^1})\\
    &\leq c_2(\|A-A^{'}\|_{H^1}+\|\Sigma-\Sigma^{'}\|_{H^1}),\\
    T_3&=(1-\alpha)|\langle \dot{A}^\top\dot{A}-\dot{A}^{{'}\top}\dot{A}^{'},K\rangle_{L^2}|\\
    &\leq (1-\alpha)c(\|(\dot{A}-\dot{A}^{'})^\top\dot{A}\|_{L^1}+\|\dot{A}^{'\top}(\dot{A}-\dot{A}^{'})\|_{L^1})\\
    &\leq c_3 \|A-A^{'}\|_{H^1},
\end{align*}
where the positive constants $c_1$, $c_2$, and $c_3$, depend only on $T$ and $U$\footnote{Also note that on $[0,T]$, the continuous embedding $L^2(0,T)\hookrightarrow L^1(0,T)$ holds:
$\|f\|_{L^1}\le \sqrt{T}\,\|f\|_{L^2}$.
}. Taking supremum over all unit vectors $(H,K)$,
\begin{align*}
   & \|D_xJ_\alpha(A,\Sigma)-D_xJ_\alpha(A^{'},\Sigma^{'})\|_{\mathcal{B}(X,\mathbb{R})}\\
    &\leq c_{T,U}\|(A-A^{'},\Sigma-\Sigma^{'})\|_{X}.
\end{align*}

Thus, the Fr\'echet derivative $D_xJ_\alpha: X\rightarrow \mathcal{B}(X,\mathbb{R})$ is locally Lipschitz, and hence $J_\alpha$ is $C^1$. Moving on to $\mathcal{C}$,
\[\|(D_x\mathcal{C}(A,\Sigma)-D_x\mathcal{C}(A^{'},\Sigma^{'}))[H,K]\|_{Y}\leq S_1+S_2,
\]
where 
\begin{align*}
    S_1&=2\|H(\Sigma-\Sigma^{'})\|_{L^2}\\
    &\leq 2\|H\|_{H^1}\|\Sigma-\Sigma^{'}\|_{C^0}\leq c_T\|\Sigma-\Sigma^{'}\|_{H^1},\\
    S_2 &= 2\|(A-A^{'})K\|_{L^2}\\
    &\leq 2\|K\|_{L^2}\|A-A^{'}\|_{C^0}\leq c_T\|A-A^{'}\|_{H^1},
\end{align*}
where $c_T$ is a positive constant depending only on $T$. Taking supremum over all unit vectors $(H,K)$ yields 
\begin{align*}
    &\|D_x\mathcal{C}(A,\Sigma)-D_x\mathcal{C}(A^{'},\Sigma^{'})\|_{\mathcal{B}(X,Y)}\\
    &\leq C_T\|(A-A^{'},\Sigma-\Sigma^{'})\|_{X}.
\end{align*}
Hence, the Fr\'echet derivative $D_x\mathcal{C}:X\rightarrow\mathcal{B}(X,Y)$ is Lipschitz, which implies that $\mathcal{C}$ is also $C^1$. 

We now verify that the Fr\'echet derivative of $\mathcal{C}$ at a  minimizer $(A^\star,\Sigma^\star)\in X$ of Problem~\ref{problem:main-min-attention} is surjective. In other words, given a triple $(R,S_0,S_T)\in Y$, we seek a pair $(H,K)\in X$ such that 
\begin{align}
&\dot{K}-A^\star K-KA^{\star\top}-H\Sigma^\star-\Sigma^\star H^\top= R~\text{ a.e. on }[0,T],\nonumber\\
    &K_0=S_0,~~K_T=S_T.\label{conds2}
\end{align}
To this end, let $\Phi$ be the fundamental matrix associated with $A^\star$, and observe that the pair $(H,K)=(0,K^0)\in X$ with $K_0^0=S_0$ and
 \begin{align*}
K^0_t=\Phi(t,0)S_0\Phi(t,0)^\top+\int_0^t\Phi(t,s)R(s)\Phi(t,s)^\top\, \text{d}s,
  \end{align*}
 already satisfies \eqref{conds2} with the exception of $K_T^0$ being off from $S_T$ by some $\Delta:= K_T^0-S_T$. We correct this by finding a pair $(H^c,K^c)\in X$ that steers $K_0^c = 0$ to $K_T^c = \Delta$ so as to not affect  $S_0$, $R$ and get exactly $S_T$. That is, the correction to $(0,K^0)$ is one that satisfies
\begin{align*}
&\dot{K}^c-A^\star K-K^cA^{\star\top}=H^c\Sigma^\star+\Sigma^\star H^{c\top}~\text{ a.e. on }[0,T],\\
    &K_c=0,~~K_T=\Delta.
\end{align*}
Setting $H^c = \frac{1}{2}F(\Sigma^\star)^{-1}$ where $F_s = \frac{1}{T}\Phi(s,T)\Delta\Phi(s,T)^\top$ readily gives
 \begin{align}\label{eq:numbered}
K^c_T=\int_0^T\Phi(T,s)F_s\Phi(T,s)^\top\, \text{d}s=\Delta.
  \end{align}
 Since $A^\star\in H^1([0,T])\hookrightarrow C^0([0,T])$, the fundamental matrix
$\Phi(\cdot,T)$ solves $\dot\Phi=A^\star\Phi$ with continuous coefficients, hence
$\Phi(\cdot,T)\in C^1$ and therefore $F_s=\tfrac1T\Phi(s,T)\Delta\Phi(s,T)^\top\in C^1([0,T])$.
By Theorem~\ref{eq:thm}, $\Sigma^\star\in H^1\hookrightarrow C^0$ and
$cI\le \Sigma_t^\star\le CI$.
Hence, the inversion map is smooth on this compact SPD set, so
$(\Sigma^\star)^{-1}\in C^0$ and
\[
\frac{d}{dt}(\Sigma^\star)^{-1}
=-(\Sigma^\star)^{-1}\dot\Sigma^\star(\Sigma^\star)^{-1}\in L^2,
\]
because $(\Sigma^\star)^{-1}\in L^\infty$ and $\dot\Sigma^\star\in L^2$.
Therefore $(\Sigma^\star)^{-1}\in H^1(0,T)$ and, since $F\in C^1$,
$H^c=\tfrac12 F(\Sigma^\star)^{-1}\in H^1(0,T;\mathbb R^{n\times n})$.
  Thus, the pair we seek is $(H,K)=(H^c,K^0+K^c)\in X$ and hence the derivative of $\mathcal{C}$ at the  minimizer is surjective. 
  
  Now, by Theorem~\ref{thm:hilbert-kkt}, a multiplier $\tilde{\Lambda}=(\Lambda,M_0,M_T)\in Y$ exists such that $D_x\mathcal{L}((A^\star,\Sigma^\star),\tilde{\Lambda})=0$, where 
  \begin{align*}
      \mathcal L(x,y)=J_\alpha(x)+\langle y,\mathcal C(x) \rangle_Y,~x\in X,~y\in Y. 
  \end{align*}
In other words, for any perturbation $(\delta A,\delta\Sigma)\in X$,
\begin{align}\label{eq:Frech}
&D_xJ_\alpha(A^\star,\Sigma^\star)[\delta A,\delta\Sigma]\nonumber\\
&
+\langle (\Lambda,M_0,M_T),\ D_x{\mathcal C}(A^\star,\Sigma^\star)[\delta A,\delta\Sigma]\rangle_{Y}=0.
\end{align}
Explicitly, the inner product in $Y$ is given by 
\begin{align*}
    &\int_0^T\hspace{-2mm}\langle \Lambda_t, \delta\dot\Sigma_t-A_t^\star\delta\Sigma_t-\delta\Sigma_t A_t^{\star\top}-\delta A_t\Sigma_t^\star-\Sigma_t^\star \delta A_t^\top\rangle_F \, \text{d}t\\
    &
+\langle M_0,\delta\Sigma_0\rangle_F+\langle M_T,\delta\Sigma_T\rangle_F.
\end{align*}
Setting {$\delta A=0$} in \eqref{eq:Frech} first, we find that for any $\delta\Sigma\in H^1$, 
\begin{align}
    &(1-\alpha)\hspace{-2mm}\int_0^T\hspace{-2mm}\mathrm{tr}(\dot A_t^{\star\top}\dot A_t^\star\delta\Sigma_t)\, \text{d}t+\langle M_0,\delta\Sigma_0\rangle_F+\langle M_T,\delta\Sigma_T\rangle_F,\nonumber
\\
&+\int_0^T\langle \Lambda_t,\delta\dot\Sigma_t-A_t^\star\delta\Sigma_t-\delta\Sigma_t A_t^{\star\top}\rangle_F\, \text{d}t=0.\label{eq:IBP}
\end{align}
Choosing $\delta\Sigma\in H_0^1([0,T];\mathbb S^n)$ so that $\delta\Sigma_0=\delta\Sigma_T=0$,
\begin{align*}
    &\int_0^T\langle \Lambda_t,\delta\dot\Sigma_t-A_t^\star\delta\Sigma_t-\delta\Sigma_t A_t^{\star\top}\rangle_F\, \text{d}t\\
    &+(1-\alpha)\int_0^T\mathrm{tr}(\dot A_t^{\star\top}\dot A_t^\star\delta\Sigma_t)
\, \text{d}t=0.
\end{align*}
By defining $\dot\Lambda\in H^{-1}$ via $\langle \dot\Lambda,\phi\rangle_{H^{-1},H_0^1}=-\int_0^T\langle \Lambda_t,\dot\phi_t\rangle_F \, \text{d}t$, we get that for all $\delta\Sigma\in H_0^1$,
\[
\langle (1-\alpha)\,\dot A^{\star\top}\dot A^\star-(\Lambda A^\star+A^{\star\top}\Lambda)-\dot\Lambda , \delta\Sigma\rangle_{H^{-1},H_0^1}=0.
\]
By the fundamental lemma in the calculus of variations (i.e., testing against smooth symmetric compactly supported variations which are dense in $H_0^1$), the distribution must vanish. Thus,  
\begin{align}\label{eq:Hminus}
    -\dot\Lambda\;\;&=\Lambda A^\star+A^{\star\top}\Lambda-(1-\alpha)\dot A^{\star\top}\dot A^\star~~\text{ in }H^{-1}.
    \end{align}
  Since $A^\star\in H^1\hookrightarrow C^0\subset L^\infty$ and $\Lambda\in L^2$,
we have $\Lambda A^\star,\,A^{\star\top}\Lambda\in L^2\hookrightarrow L^1$,
and $\dot A^{\star\top}\dot A^\star\in L^1$, so that the right-hand side of \eqref{eq:Hminus} is in $L^1([0,T];\mathbb{S}^n)$. Thus $\Lambda\in W^{1,1}([0,T];\mathbb{S}^n)$, yielding \eqref{eq:adjoint}. Now, with $\Lambda\in W^{1,1}$, for any $\delta \Sigma\in H^1$, integration by parts in \eqref{eq:IBP} gives $\langle M_0-\Lambda_0,\delta\Sigma_0\rangle_F+\langle M_T+\Lambda_T,\delta\Sigma_T\rangle_F=0$, from which \eqref{eq:trans} follows.

Next, setting {$\delta\Sigma = 0$} in \eqref{eq:Frech}, we obtain for any $\delta A\in H_0^1$,
\begin{align}
    &2\alpha\int_0^T\langle A_t^\star,\delta A_t\rangle_F\, \text{d}t
+2(1-\alpha)\int_0^T\langle \dot{A}_t^\star\Sigma_t^\star, \delta\dot{ A}_t\rangle_F\, \text{d}t\nonumber
\\
&-2\int_0^T\langle \Lambda_t\Sigma_t^\star,\delta A_t\rangle_F\, \text{d}t=0.\label{eqf}
\end{align}
Again, defining $\frac{d}{dt}(\dot{A}^\star\Sigma^\star)\in H^{-1}$ in the distributional sense,  $$\langle \frac{d}{dt}(\dot{A}^\star\Sigma^\star),\phi\rangle_{H^{-1},H_0^1}=-\int_0^T\langle \dot{A}_t^\star\Sigma_t^\star,\dot\phi_t\rangle_F \, \text{d}t,$$ we obtain by the fundamental Lemma, that 
\begin{align*}
   \Lambda\Sigma^\star= \alpha A^\star-(1-\alpha)\frac{d}{dt}(\dot{A}^\star\Sigma^\star)~~\text{ in }H^{-1}.
\end{align*}
Since both $\Lambda\Sigma^\star$ and $A^\star$ are in $L^1$, so is $ \frac{d}{dt}(\dot{A}^\star\Sigma^\star)$ and therefore $\dot{A}^\star\Sigma^\star\in W^{1,1}$. Thus, the stationarity condition \eqref{eq:stationarity} holds almost everywhere. Finally, for any $\delta A\in H^1$, integration by parts in \eqref{eqf} gives 
\begin{align*}
\Big[\,2(1-\alpha)\,\langle \dot A^\star\Sigma^\star,\ \delta A\rangle_F\Big]_0^T=0,
\end{align*}
and since $\delta A_0$ and $\delta A_T$ are arbitrary and $\Sigma^\star$ in invertible, we obtain 
\eqref{eq:BC-A}. The primal equation \eqref{eq:primal} holds by feasibility. This completes the proof.
\end{proof}

Next, we take advantage of the structure in the first-order necessary conditions to upgrade the regularity of the  minimizer found in Theorem~\ref{eq:thm}.

\subsection{Regularity of minimizers}
In this section, we exploit  Theorem~\ref{thm:KKT-aug} together with the
uniform interior bounds from Theorem~\ref{eq:thm} to upgrade the regularity of any 
minimizer. We begin with an initial Sobolev lift, and then prove that minimizers are smooth in time.

\begin{lemma}[Initial regularity upgrade]\label{thm:reg-upgrade}
Let $(A^\star,\Sigma^\star)\in\mathcal{H}$ be a  minimizer of
Problem~\ref{problem:main-min-attention} and let $\Lambda^\star$ be the multiplier from
Theorem~\ref{thm:KKT-aug}. Then
\begin{enumerate}
\item[\textnormal{(i)}] $\Lambda^\star\in W^{1,1}([0,T];\mathbb S^n)\hookrightarrow C^0([0,T];\mathbb S^n)$,
\item[\textnormal{(ii)}] $A^\star\in W^{2,1}([0,T];\mathbb R^{n\times n})\hookrightarrow C^{1}([0,T];\mathbb R^{n\times n})$,
\item[\textnormal{(iii)}] $\Sigma^\star\in W^{3,1}([0,T];\mathbb S^n)\hookrightarrow C^{2}([0,T];\mathbb S^n)$,
\end{enumerate}
where $C^{k}([0,T],\cdot)$ denotes the space of $k$-times continuously differentiable functions.
\end{lemma}

\begin{proof}
We proceed in three steps. Throughout, we use that on the finite interval $[0,T]$ one has
$L^2\subset L^1$ and that $H^1([0,T])\hookrightarrow C^0([0,T])$.

\smallskip\noindent
\emph{Step 1: $\Lambda^\star\in W^{1,1}$.}
The adjoint equation \eqref{eq:adjoint} holds in $H^{-1}$:
\[
-\dot\Lambda^\star
=\Lambda^\star A^\star+A^{\star\top}\Lambda^\star-(1-\alpha)\dot A^{\star\top}\dot A^\star .
\]
Since $A^\star\in H^1\hookrightarrow C^0$ and $\Lambda^\star\in L^2$,
the products $\Lambda^\star A^\star$ and $A^{\star\top}\Lambda^\star$ lie in $L^1$.
Also $\dot A^{\star\top}\dot A^\star\in L^1$ because $\dot A^\star\in L^2$.
Hence the right-hand side is in $L^1$, so $\dot\Lambda^\star\in L^1$ and thus
$\Lambda^\star\in W^{1,1}\hookrightarrow C^0$, proving (i); this is already implicit in Theorem~\ref{thm:KKT-aug}, but we record it here for completeness.

\smallskip\noindent
\emph{Step 2: $A^\star\in W^{2,1}$.}
The stationarity condition \eqref{eq:stationarity} holds in $H^{-1}$:
\begin{equation}\label{eq:stat-weak}
\Lambda^\star\Sigma^\star
= \alpha\,A^\star-(1-\alpha)\tfrac{d}{dt}\big(\dot A^\star\Sigma^\star\big).
\end{equation}
By Step~1, $\Lambda^\star\in W^{1,1}\hookrightarrow L^\infty$; also
$A^\star,\Sigma^\star\in H^1\hookrightarrow C^0$. Hence the right-hand side of
\eqref{eq:stat-weak} belongs to $L^1$, so
\[
\tfrac{d}{dt}\big(\dot A^\star\Sigma^\star\big)\in L^1([0,T];\mathbb R^{n\times n}),
~~
f:=\dot A^\star\Sigma^\star\in W^{1,1}\hookrightarrow C^0 .
\]
The uniform bounds $cI\le \Sigma^\star_t\le CI$ from Theorem~\ref{eq:thm} imply
$(\Sigma^\star)^{-1}\in L^\infty$. Differentiating the identity
$f=\dot A^\star\Sigma^\star$ in $\mathcal D'(0,T)$\footnote{The space $\mathcal D'(0,T)$ is the space of distributions, that is, the space of continuous linear functionals on the space of smooth functions with compact support.} gives
\[
\dot f=\ddot A^\star\Sigma^\star+\dot A^\star\dot\Sigma^\star\in L^1,
\]
where $\dot A^\star,\dot\Sigma^\star\in L^2\subset L^1$.
Multiplying by $(\Sigma^\star)^{-1}$ yields
\[
\ddot A^\star = \big(\dot f-\dot A^\star\dot\Sigma^\star\big)(\Sigma^\star)^{-1}\in L^1.
\]
Therefore, $A^\star\in W^{2,1}\hookrightarrow C^1$, proving (ii).

\smallskip\noindent
\emph{Step 3: $\Sigma^\star\in W^{3,1}$.}
Differentiate the Lyapunov equation \eqref{eq:primal} in $\mathcal D'(0,T)$:
\[
\ddot\Sigma^\star
= \dot A^\star\Sigma^\star
+ A^\star\dot\Sigma^\star
+ \dot\Sigma^\star A^{\star\top}
+ \Sigma^\star\dot A^{\star\top}.
\]
Here $A^\star,\Sigma^\star\in C^0\subset L^\infty$ and
$\dot A^\star,\dot\Sigma^\star\in L^2\subset L^1$, so each product lies in $L^1$.
Hence $\ddot\Sigma^\star\in L^1$, giving $\Sigma^\star\in W^{2,1}$.
Differentiate once more in $\mathcal D'(0,T)$ and use $\ddot A^\star\in L^1$
from Step~2:
\begin{align*}
\Sigma^{\star(3)}
&= \ddot A^\star\Sigma^\star
+ 2\,\dot A^\star\dot\Sigma^\star
+ A^\star\ddot\Sigma^\star\\
&+ \ddot\Sigma^\star A^{\star\top}
+ 2\,\dot\Sigma^\star\dot A^{\star\top}
+ \Sigma^\star\ddot A^{\star\top}
\in L^1,
\end{align*}
since $L^2\cdot L^2\subset L^1$ and $L^\infty\cdot L^1\subset L^1$ on $[0,T]$.
Thus $\Sigma^\star\in W^{3,1}\hookrightarrow C^2$. This proves (iii) and completes the proof.
\end{proof}

\begin{thm}[Smoothness in time of minimizers]\label{thm:reg-analytic}
Let $(A^\star,\Sigma^\star)\in\mathcal H$ be a  minimizer of
Problem~\ref{problem:main-min-attention} with multiplier $\Lambda^\star$ from
Theorem~\ref{thm:KKT-aug}. Then $A^\star,\Sigma^\star,\Lambda^\star$ admit
representatives that are \emph{smooth} on $[0,T]$. In particular,
\[
A^\star,\Sigma^\star,\Lambda^\star \in C^\infty([0,T]).
\]
\end{thm}

\begin{proof}
By Theorem~\ref{eq:thm} there exist $0<c\le C<\infty$ such that
\begin{equation}\label{eq:SPD-bounds}
cI \le \Sigma_t^\star \le CI \qquad \forall t\in[0,T].
\end{equation}
Consequently $\Sigma_t^\star$ is invertible for all $t$, and $(\Sigma^\star)^{-1}\in L^\infty([0,T];\mathbb S^n)$. We are now ready to prove by induction that for every integer $m\ge 1$, we have
\begin{align}\label{eq:bootstrap-claim}
    &\Lambda^\star\in W^{m,1}([0,T]),\nonumber\\
&A^\star\in W^{m+1,1}([0,T]),\\
&\Sigma^\star\in W^{m+2,1}([0,T]).\nonumber
\end{align}
The base case $m=1$ is exactly Lemma~\ref{thm:reg-upgrade}. For the inductive step, assume \eqref{eq:bootstrap-claim} holds for some $m\ge 1$.
Since in one dimension, 
\[
W^{r,1}([0,T])\hookrightarrow C^{r-1}([0,T])\quad\text{for all }r\ge 1,
\]
the inductive hypothesis implies
\[
\Lambda^\star\in C^{m-1},\qquad
A^\star\in C^{m},\qquad
\Sigma^\star\in C^{m+1},
\]
so all derivatives up to these orders are bounded on $[0,T]$.
We repeatedly use the distributional Leibniz rule: if $f,g\in W^{m,1}$, then
\begin{equation}\label{eq:leibniz}
(fg)^{(m)}=\sum_{\ell=0}^m \binom{m}{\ell} f^{(\ell)}g^{(m-\ell)}
\quad\text{in }\mathcal D'(0,T).
\end{equation}

\smallskip\noindent
\emph{(a) Upgrade $\Lambda^\star$ from $W^{m,1}$ to $W^{m+1,1}$.}
By Theorem~\ref{thm:KKT-aug}, the adjoint equation holds in $\mathcal D'(0,T)$:
\[
-\dot\Lambda^\star
=\Lambda^\star A^\star+A^{\star\top}\Lambda^\star
-(1-\alpha)\dot A^{\star\top}\dot A^\star .
\]
Differentiate $m$ times distributionally. Using \eqref{eq:leibniz},
\begin{align*}
(\Lambda^\star A^\star)^{(m)}
&=\sum_{\ell=0}^m \binom{m}{\ell}\Lambda^{\star(\ell)}A^{\star(m-\ell)},\\
(A^{\star\top}\Lambda^\star)^{(m)}
&=\sum_{\ell=0}^m \binom{m}{\ell}A^{\star(m-\ell)\top}\Lambda^{\star(\ell)},\\
(\dot A^{\star\top}\dot A^\star)^{(m)}
&=\sum_{\ell=0}^m \binom{m}{\ell}A^{\star(\ell+1)\top}A^{\star(m-\ell+1)}.
\end{align*}
In each sum, at most one factor is of top order in $L^1$
(namely $\Lambda^{\star(m)}$, $A^{\star(m)}$, or $A^{\star(m+1)}$),
while the remaining factors are lower–order and hence bounded. Thus
every term lies in $L^1([0,T])$, so
$\Lambda^{\star(m+1)}\in L^1$, i.e., $\Lambda^\star\in W^{m+1,1}$.

\smallskip\noindent
\emph{(b) Upgrade $A^\star$ from $W^{m+1,1}$ to $W^{m+2,1}$.}
The stationarity condition holds in $\mathcal D'(0,T)$:
\begin{equation}\label{eq:stat-boot}
\Lambda^\star\Sigma^\star
= \alpha A^\star-(1-\alpha)\frac d{dt}\big(\dot A^\star\Sigma^\star\big).
\end{equation}
Since $\alpha\in(0,1)$, rearrange to
\begin{equation}\label{eq:fprime}
\frac d{dt}\big(\dot A^\star\Sigma^\star\big)
= \frac{\alpha}{1-\alpha}A^\star
-\frac{1}{1-\alpha}\Lambda^\star\Sigma^\star .
\end{equation}
Differentiate \eqref{eq:fprime} $m$ times in $\mathcal D'(0,T)$ and apply
\eqref{eq:leibniz} to $(\Lambda^\star\Sigma^\star)^{(m)}$:
\[
(\Lambda^\star\Sigma^\star)^{(m)}
=\sum_{\ell=0}^m \binom{m}{\ell}\Lambda^{\star(\ell)}\Sigma^{\star(m-\ell)}.
\]
By part (a), $\Lambda^\star\in W^{m+1,1}$, and by the inductive hypothesis
$A^\star\in W^{m+1,1}$, $\Sigma^\star\in W^{m+2,1}$, so the right–hand side of
\eqref{eq:fprime} has $m$-th derivative in $L^1$. Hence
\[
\frac d{dt}\big(\dot A^\star\Sigma^\star\big)\in W^{m,1},
\quad\text{equivalently}\quad
\dot A^\star\Sigma^\star\in W^{m+1,1}.
\]
Set $f:=\dot A^\star\Sigma^\star\in W^{m+1,1}$. Differentiating
$f=\dot A^\star\Sigma^\star$ $(m+1)$ times and using \eqref{eq:leibniz} yields
\[
f^{(m+1)}
= A^{\star(m+2)}\Sigma^\star
+\sum_{\ell=0}^{m+1}\binom{m+1}{\ell}\,
A^{\star(\ell+1)}\Sigma^{\star(m+1-\ell)} .
\]
Again, each term in the sum is in $L^1$ since at most one factor is top order
and all others are bounded. Using \eqref{eq:SPD-bounds} so that
$(\Sigma^\star)^{-1}\in L^\infty$, we solve for
\begin{align*}
   &A^{\star(m+2)}
=\\
&\Bigl(f^{(m+1)}-\sum_{\ell=0}^{m+1}\binom{m+1}{\ell}\,
A^{\star(\ell+1)}\Sigma^{\star(m+1-\ell)}\Bigr)(\Sigma^\star)^{-1}\in L^1.
\end{align*}
Thus $A^\star\in W^{m+2,1}$.

\smallskip\noindent
\emph{(c) Upgrade $\Sigma^\star$ from $W^{m+2,1}$ to $W^{m+3,1}$.}
The Lyapunov equation holds in $\mathcal D'(0,T)$:
\begin{equation}\label{eq:lyap-boot}
\dot\Sigma^\star = A^\star\Sigma^\star+\Sigma^\star A^{\star\top}+BB^\top .
\end{equation}
Differentiate \eqref{eq:lyap-boot} $(m+2)$ times distributionally. Applying
\eqref{eq:leibniz} gives
\begin{align*}
(A^\star\Sigma^\star)^{(m+2)}
&=\sum_{\ell=0}^{m+2}\binom{m+2}{\ell}
A^{\star(\ell)}\Sigma^{\star(m+2-\ell)},\\
(\Sigma^\star A^{\star\top})^{(m+2)}
&=\sum_{\ell=0}^{m+2}\binom{m+2}{\ell}
\Sigma^{\star(\ell)}A^{\star(m+2-\ell)\top}.
\end{align*}
By part (b) we have $A^\star\in W^{m+2,1}$, and by the inductive hypothesis
$\Sigma^\star\in W^{m+2,1}$; hence each product above contains at most one
top–order $L^1$ factor and all remaining factors bounded. Therefore
$\Sigma^{\star(m+3)}\in L^1$, i.e., $\Sigma^\star\in W^{m+3,1}$.

\medskip\noindent
Combining (a)--(c) yields \eqref{eq:bootstrap-claim} for $m+1$.
By induction, \eqref{eq:bootstrap-claim} holds for all $m$.
Consequently $A^\star,\Sigma^\star,\Lambda^\star\in W^{k,1}$ for every $k$,
and thus $A^\star,\Sigma^\star,\Lambda^\star\in C^\infty([0,T])$.
\end{proof}


\section{Spatial attention}\label{sec:spatial}

In this section, we study the edge case $\alpha=1$, where the attention functional
penalizes only the spatial component. The resulting variational
problem seeks a control that steers $\Sigma_0$ to $\Sigma_T$ while minimizing
\[
J_1(A,\Sigma):=\int_0^T \mathrm{tr}(A_tA_t^\top)\,dt=\int_0^T\|A_t\|_F^2\,dt .
\]
Unlike the mixed case $0<\alpha<1$, no temporal penalty on $\dot A$ is present.
Nevertheless, the structure of the first-order necessary conditions together with uniform interior spectral bounds still
forces minimizers to be smooth in time.

\begin{problem}[Spatial attention: $\alpha=1$]\label{problem:case1-min-attention}
Given $T>0$, $B\in\mathbb R^{n\times m}$, and endpoints
$\Sigma_{\init},\Sigma_{\fin}\in\mathbb S_{++}^n$, determine
\[
\min_{(A,\Sigma)\in\mathcal H_1} J_1(A,\Sigma),
\]
subject to the Lyapunov dynamics
\begin{equation}\label{eq:constr1-a1}
\dot\Sigma_t=A_t\Sigma_t+\Sigma_tA_t^\top+BB^\top ~~ \text{a.e. on }[0,T],
\end{equation}
and endpoint constraints
\begin{equation}\label{eq:constr3}
\Sigma_0=\Sigma_{\init},\qquad \Sigma_T=\Sigma_{\fin}.
\end{equation}
Here, the search space is
\[
\mathcal H_1 := L^2([0,T];\mathbb R^{n\times n})\oplus H^1([0,T];\mathbb S^n) .
\]
\end{problem}
In what follows many of the steps in the proofs are similar to those in the ealier section and will be given with less detail, to avoid repetition. 

\subsection{Existence of minimizers}

\begin{thm}[Existence in the spatial case $\alpha=1$]\label{thm:exist-alpha1}
Problem~\ref{problem:case1-min-attention} admits a  minimizer
$(A^\star,\Sigma^\star)\in\mathcal H_1$.
Moreover, every  minimizer $\Sigma^\star$ satisfies the uniform spectral bounds
\[
c\,I \le \Sigma_t^\star \le C\,I,\qquad \forall t\in[0,T],
\]
where, writing $J_1^\star:=\inf J_1$,
\begin{align*}
c&:=e^{-2\sqrt{T\,J_1^\star}}\,\lambda_{\min}(\Sigma_{\init}),\\
C&:=e^{2\sqrt{T\,J_1^\star}}
\Bigl(\lambda_{\max}(\Sigma_{\init})+T\,\lambda_{\max}(BB^\top)\Bigr).
\end{align*}
\end{thm}

\begin{proof}
\emph{Step 1: Feasibility and finite infimum.}
Consider the reference covariance path
\[
\Sigma_t:=(1-t/T)\Sigma_{\init}+(t/T)\Sigma_{\fin},\qquad t\in[0,T].
\]
Since $\Sigma_{\init},\Sigma_{\fin}\in\mathbb S_{++}^n$ and $\mathbb S_{++}^n$ is convex,
$\Sigma_t\in\mathbb S_{++}^n$ for all $t$. Define
\[
A_t:=\tfrac12(\dot\Sigma_t-BB^\top)(\Sigma_t)^{-1}.
\]
Then $\Sigma\in H^1$ and $A\in C^0\subset L^2$, and by direct substitution
$(A,\Sigma)$ satisfies \eqref{eq:constr1-a1}--\eqref{eq:constr3}.
Hence, the feasible set is nonempty and
\[
0\le \inf J_1 \le J_1(A,\Sigma)<\infty.
\]

\emph{Step 2: Minimizing sequence and coercivity.}
Let $(A^{k},\Sigma^{k})\in\mathcal H_1$ be a minimizing sequence with
\begin{equation}\label{eq:M-a1}
J_1(A^{k},\Sigma^{k})\le J_1^\star+\tfrac1k=:M_k .
\end{equation}
Then $\|A^k\|_{L^2}^2\le M_k$ and by Cauchy--Schwarz,
\begin{equation}\label{eq:L1bound-a1}
\|A^k\|_{L^1}\le \sqrt{T}\,\|A^k\|_{L^2}\le \sqrt{T\,M_k}.
\end{equation}

\emph{Step 3: Uniform spectral bounds on $\Sigma^k$.}
Let $\Phi^k(t,s)$ solve $\partial_t\Phi^k(t,s)=A_t^k\Phi^k(t,s)$ with $\Phi^k(s,s)=I$.
Gronwall and \eqref{eq:L1bound-a1} yield
\begin{align*}
    &\sigma_{\max}(\Phi^k(t,s))\le \exp\!\Big(\int_s^t\|A_u^k\|\,du\Big)
\le e^{\sqrt{T\,M_k}},\\
&
\sigma_{\min}(\Phi^k(t,s))\ge e^{-\sqrt{T\,M_k}}.
\end{align*}
Using the representation
\[
\Sigma_t^k=\Phi^k(t,0)\Sigma_{\init}\Phi^k(t,0)^\top
+\int_0^t\Phi^k(t,s)BB^\top\Phi^k(t,s)^\top\,ds,
\]
Rayleigh quotient bounds give, for all $t$,
\[
\lambda_{\min}(\Sigma_t^k)\ge e^{-2\sqrt{T\,M_k}}\lambda_{\min}(\Sigma_{\init})=:c_k,
\]
and
\[
\lambda_{\max}(\Sigma_t^k)\le
e^{2\sqrt{T\,M_k}}\Bigl(\lambda_{\max}(\Sigma_{\init})
+T\,\lambda_{\max}(BB^\top)\Bigr)=:C_k.
\]
Since $M_k\le M_1$ for all $k$, we obtain uniform bounds
\begin{equation}\label{eq:sbounds-a1}
0<c_1 I \le \Sigma_t^k \le C_1 I<\infty .
\end{equation}

\emph{Step 4: $H^1$--bounds for $\Sigma^k$.}
From \eqref{eq:constr1-a1} and \eqref{eq:sbounds-a1},
\[
\dot\Sigma^k=A^k\Sigma^k+\Sigma^kA^{k\top}+BB^\top\in L^2,
\]
because $\Sigma^k\in L^\infty$, $A^k\in L^2$ and \eqref{eq:L1bound-a1}. Hence $\Sigma^k$ is uniformly bounded
in $H^1([0,T])$.

\emph{Step 5: Compactness and passage to the limit.}
By reflexivity and Banach--Alaoglu there is a subsequence (not relabeled) and
$(A^\star,\Sigma^\star)\in\mathcal H_1$ such that
\[
A^k\rightharpoonup A^\star\ \text{ in }L^2,\qquad
\Sigma^k\rightharpoonup \Sigma^\star\ \text{ in }H^1 .
\]
Rellich--Kondrachov in one dimension gives $\Sigma^k\to\Sigma^\star$ in $C^0$.
Thus, endpoints \eqref{eq:constr3} pass to the limit.
Moreover,
\[
A^k\Sigma^k-A^\star\Sigma^\star=(A^k-A^\star)\Sigma^\star+A^k(\Sigma^k-\Sigma^\star),
\]
where the first term converges weakly to $0$ in $L^2$, and the second tends to $0$ in $L^2$
because $\|A^k\|_{L^2}$ is bounded and $\Sigma^k\to\Sigma^\star$ uniformly.
The same holds for $\Sigma^kA^{k\top}$, so \eqref{eq:constr1-a1} holds for the limit.
Hence $(A^\star,\Sigma^\star)$ is feasible.

\emph{Step 6: Minimality and limiting spectral bounds.}
Weak lower semicontinuity of $\|\cdot\|_{L^2}^2$ yields
\[
J_1(A^\star,\Sigma^\star)\le \liminf_k J_1(A^k,\Sigma^k)=J_1^\star .
\]
Thus $(A^\star,\Sigma^\star)$ is a  minimizer.
Finally, Weyl’s inequalities combined with the uniform convergence of $\Sigma^k$
pass \eqref{eq:sbounds-a1} to the limit, giving the stated constants with
$M_k\to J_1^\star$.
\end{proof}

\begin{remark}[Relaxed bounds]
As in the mixed case, any admissible value $\bar J_1\ge J_1^\star$ may be substituted for
$J_1^\star$ in the expressions of $c$ and $C$, yielding conservative computable
bounds.
\end{remark}

\subsection{First--order necessary conditions}

\begin{thm}[FONC for the spatial case]\label{thm:KKT-alpha1}
Let $(A^\star,\Sigma^\star)\in\mathcal H_1$ be a  minimizer of
Problem~\ref{problem:case1-min-attention}. Then there exist multipliers
$\Lambda\in W^{1,1}([0,T];\mathbb S^n)$ and $M_0,M_T\in\mathbb S^n$ such that
\begin{subequations}\label{eq:KKT-alpha1}
\begin{align}
\dot\Sigma^\star
&=A^\star\Sigma^\star+\Sigma^\star A^{\star\top}+BB^\top,
\label{eq:KKT-a1-primal}\\
-\dot\Lambda
&=\Lambda A^\star+A^{\star\top}\Lambda,
\label{eq:KKT-a1-adjoint}\\
A^\star&=\Lambda\Sigma^\star,
\label{eq:KKT-a1-stationarity}\\
\Lambda_0&=M_0,\qquad \Lambda_T=-M_T .
\label{eq:KKT-a1-trans}
\end{align}
\end{subequations}
All identities hold a.e.\ on $[0,T]$.
\end{thm}

\begin{proof}
Apply Theorem~\ref{thm:hilbert-kkt} with $X=\mathcal H_1$,
$Y=L^2([0,T];\mathbb S^n)\oplus\mathbb S^n\oplus\mathbb S^n$,
constraint map $\mathcal C$ as in the mixed case, and $J_1(A,\Sigma)=\|A\|_{L^2}^2$.
The maps $J_1$ and $\mathcal C$ are Fr\'echet $C^1$ on $X$ by the same arguments as in
Theorem~\ref{thm:KKT-aug} (here the $J_1$ derivative simplifies). The surjectivity of
$D\mathcal C(A^\star,\Sigma^\star)$ follows identically, and is in fact easier since
only $H^c\in L^2$ is required (for definion of $H^c$ see text before \eqref{eq:numbered}).

Thus there exists $\tilde\Lambda=(\Lambda,M_0,M_T)\in Y$ such that
$D_x\mathcal L((A^\star,\Sigma^\star),\tilde\Lambda)=0$.
Variations $\delta\Sigma\in H_0^1$ give the adjoint equation in $H^{-1}$ and, since
$\Lambda,A^\star\in L^2$ and $L^2\cdot L^2\subset L^1$ on $[0,T]$, we conclude
$\dot\Lambda\in L^1$, so $\Lambda\in W^{1,1}$ and \eqref{eq:KKT-a1-adjoint} holds a.e.
Boundary variations yield \eqref{eq:KKT-a1-trans}.
Variations $\delta A\in L^2$ give
\[
0=\int_0^T\!\langle 2A^\star-2\Lambda\Sigma^\star,\delta A\rangle_F\,dt
\quad\forall\,\delta A\in L^2,
\]
hence \eqref{eq:KKT-a1-stationarity} holds a.e.
The primal equation \eqref{eq:KKT-a1-primal} holds by feasibility.
\end{proof}

\subsection{Regularity of minimizers}

\begin{thm}[Smoothness in the spatial case]\label{thm:reg-alpha1}
Let $(A^\star,\Sigma^\star)$ be a  minimizer of
Problem~\ref{problem:case1-min-attention} with multiplier $\Lambda$ from
Theorem~\ref{thm:KKT-alpha1}. Then
\[
A^\star,\Sigma^\star,\Lambda\in C^\infty([0,T]).
\]
\end{thm}

\begin{proof}
\emph{Step 1: Initial lift to continuity.}
From Theorem~\ref{thm:exist-alpha1}, $\Sigma^\star\in H^1([0,T])\hookrightarrow C^0$
and satisfies $cI\le \Sigma_t^\star\le CI$.
By Theorem~\ref{thm:KKT-alpha1},
$A^\star\in L^2$ and $\Lambda\in L^2$ with
\[
-\dot\Lambda=\Lambda A^\star+A^{\star\top}\Lambda \quad\text{ a.e. }H^{-1}.
\]
Since $\Lambda,A^\star\in L^2$ and $L^2\cdot L^2\subset L^1$,
the right--hand side lies in $L^1$, hence $\dot\Lambda\in L^1$ and
$\Lambda\in W^{1,1}\hookrightarrow C^0$.
Stationarity $A^\star=\Lambda\Sigma^\star$ then gives $A^\star\in C^0$.

\emph{Step 2: Classical $C^1$ regularity.}
With $A^\star,\Sigma^\star\in C^0$, the primal equation gives
$\dot\Sigma^\star\in C^0$, hence $\Sigma^\star\in C^1$.
Similarly, the adjoint equation gives $\dot\Lambda\in C^0$, so $\Lambda\in C^1$.
Therefore $A^\star=\Lambda\Sigma^\star\in C^1$.

\emph{Step 3: Bootstrap to $C^\infty$.}
Differentiate the pointwise FONC system repeatedly. At each order $m\ge 0$,
Leibniz' rule yields expressions for $\Sigma^{\star\,(m+1)}$ and $\Lambda^{(m+1)}$
as finite sums of products of derivatives of order at most $m$ of
$\Sigma^\star$ and $\Lambda$ (and the constant matrix $BB^\top$).
Since the right-hand sides are smooth combinations of already-smooth terms,
standard ODE bootstrapping implies
$\Sigma^\star,\Lambda\in C^{m+1}$ whenever they are in $C^{m}$.
By induction, $\Sigma^\star,\Lambda\in C^\infty$, and then
$A^\star=\Lambda\Sigma^\star\in C^\infty$ as well.
\end{proof}

\begin{remark}[Comparison with $0<\alpha<1$]
In the mixed case, smoothness arises from a Sobolev bootstrap driven by the
$\dot A$--penalty. For $\alpha=1$ there is no coercive control on $\dot A$ at the energy
level, but the algebraic stationarity $A=\Lambda\Sigma$ couples the primal and adjoint
flows into a polynomial ODE system, which still forces $C^\infty$ regularity once
$\Sigma$ stays uniformly SPD.
\end{remark}

\subsection{Zero--noise limit}
\label{subsec:zero-noise}

We now specialize to the case where $B\equiv0$. It is of interest since the FONC system becomes autonomous and admits explicit solutions.

\begin{prop}[FONC structure and closed forms]\label{prop:ZN-structure}
Assume $B\equiv0$. Let $(A^\star,\Sigma^\star,\Lambda^\star)$ satisfy
\eqref{eq:KKT-alpha1}, i.e.,
\begin{align*}
\dot\Sigma^\star&=A^\star\Sigma^\star+\Sigma^\star A^{\star\top},\\
-\dot\Lambda^\star&=\Lambda^\star A^\star+A^{\star\top}\Lambda^\star,\\
A^\star&=\Lambda^\star\Sigma^\star .
\end{align*}
Write $A^\star_t=S_t+\Omega_t$ with
$S_t=\tfrac12(A^\star_t+A^{\star\top}_t)\in\mathbb S^n$ and
$\Omega_t=\tfrac12(A^\star_t-A^{\star\top}_t)\in\mathbb A^n$.
Then:
\begin{enumerate}
\item[\textnormal{(i)}] $\Omega$ is constant in time and $\tr A^\star_t\equiv\tr A^\star_0$.
\item[\textnormal{(ii)}] $S$ solves $\dot S_t=2(\Omega S_t-S_t\Omega)$, hence
\begin{equation}\label{eq:A-closed}
A^\star_t=e^{2\Omega t}\,S_0\,e^{-2\Omega t}+\Omega .
\end{equation}
\item[\textnormal{(iii)}] Let 
$Z_t:=e^{-2\Omega t}\Sigma^\star_te^{2\Omega t}$. Then
\[
\dot Z_t=A_0^{\star\top}Z_t+Z_t A_0^\star,\qquad
Z_t=e^{A_0^{\star\top}t}\Sigma_0\,e^{A_0 ^\star t},
\]
so
\begin{equation}\label{eq:Sigma-closed}
\Sigma^\star_t=e^{2\Omega t}\,e^{A_0^{\star\top}t}\,\Sigma_0\,e^{A_0^\star t}\,e^{-2\Omega t},
\end{equation}
and in particular
\begin{equation}\label{eq:SigmaT-closed}
\Sigma_T^\star
= e^{2\Omega T}\,e^{A_0^{\star\top}T}\,\Sigma_0\,e^{A_0 ^\star T}\,e^{-2\Omega T}.
\end{equation}
\end{enumerate}
\end{prop}

\begin{proof}
Differentiate $A^\star=\Lambda^\star\Sigma^\star$ and substitute the primal and adjoint
equations:
\begin{align*}
\dot A^\star
&=\dot\Lambda^\star\Sigma^\star+\Lambda^\star\dot\Sigma^\star\\
&=-(\Lambda^\star A^\star+A^{\star\top}\Lambda^\star)\Sigma^\star
+\Lambda^\star(A^\star\Sigma^\star+\Sigma^\star A^{\star\top})\\
&=-A^{\star\top}(\Lambda^\star\Sigma^\star)+(\Lambda^\star\Sigma^\star)A^{\star\top}\\
&=A^\star A^{\star\top}-A^{\star\top}A^\star .
\end{align*}
The right-hand side is symmetric and traceless, so the skew part
$\Omega=\tfrac12(A^\star-A^{\star\top})$ is constant and
$\tfrac d{dt}\tr A^\star=0$, proving (i).
Taking symmetric parts yields $\dot S=2(\Omega S-S\Omega)$, whose solution is
$S_t=e^{2\Omega t}S_0 e^{-2\Omega t}$, hence \eqref{eq:A-closed}.
For (iii), define $Z_t=e^{-2\Omega t}\Sigma^\star_t e^{2\Omega t}$ and differentiate;
using \eqref{eq:A-closed} one obtains $\dot Z_t=A_0^{^\star\top}Z_t+Z_tA_0^\star$ with explicit solution
$Z_t=e^{A_0^{\star\top}t}\Sigma_0e^{A_0^\star t}$, which yields
\eqref{eq:Sigma-closed}--\eqref{eq:SigmaT-closed}.
\end{proof}

\section{Temporal attention}\label{sec:temporal}

We now consider the limiting regime $\alpha=0$, in which the attention cost reduces to
\[
J_0(A,\Sigma)=\int_0^T \mathrm{tr}(\dot A_t\,\Sigma_t\,\dot A_t^\top)\,dt .
\]
Since $J_0$ depends only on $\dot A$, every \emph{constant} control $A$ that steers the
covariance from $\Sigma_0$ to $\Sigma_T$ is a  minimizer of $J_0$ (indeed
$J_0=0$ whenever $\dot A\equiv0$). The temporal problem is therefore degenerate:
the set of minimizers coincides with the set of constant feasible controls.

To extract a meaningful selection among these constant minimizers, we adopt a standard
perturbative viewpoint. For $\varepsilon\in(0,1)$, consider instead
\[
J_{\varepsilon}(A,\Sigma)=\varepsilon\!\int_0^{T}\!\!\mathrm{tr}(A_tA_t^{\top})\,dt
+(1-\varepsilon)\!\int_0^{T}\!\!\mathrm{tr}(\dot A_t\Sigma_t\dot A_t^{\top})\,dt .
\]
As $\varepsilon\downarrow0$, the temporal term dominates and optimal controls flatten in
time. At the next order, one expects selection among constant feasible controls by the
spatial energy $\mathrm{tr}(AA^\top)$. This motivates the reduced finite-dimensional
problem.

\begin{problem}[Temporally--selected constant control]\label{prob:temporal-const}
Given $T>0$, $B\in\mathbb{R}^{n\times m}$, and
$\Sigma_{\init},\Sigma_{\fin}\in\mathbb{S}_{++}^n$, determine
\[
\min_{A\in\mathbb{R}^{n\times n}}\ \mathrm{tr}(AA^{\top})
\]
subject to the Lyapunov dynamics with \emph{constant} $A$,
\begin{equation}\label{eq:const-lyap}
\dot\Sigma_t=A\Sigma_t+\Sigma_tA^{\top}+BB^{\top},~~
\Sigma_0=\Sigma_\init,~~ \Sigma_{T}=\Sigma_\fin .
\end{equation}
We denote by $\mathcal{F}_{\mathrm{const}}$ the (possibly empty) set of constant
matrices $A$ for which \eqref{eq:const-lyap} satisfies the endpoint constraint.
\end{problem}
\subsection{Existence of minimizers}

\begin{thm}[Existence of minimizers]
\label{thm:temporal-selection}
Let $\alpha=\varepsilon\in(0,1)$ and let $(A_\varepsilon,\Sigma_\varepsilon)$ be 
minimizers of $J_{\varepsilon}$ subject to the Lyapunov dynamics and endpoints.
Assume $\mathcal F_{\rm const}\neq \emptyset$. Then there is a sequence
$\varepsilon_k\downarrow0$ and $A^0\in\mathcal F_{\rm const}$ such that
\[
A_{\varepsilon_k}\to A^0 \ \text{in } C^0([0,T]),\qquad
\Sigma_{\varepsilon_k}\to \Sigma^0 \ \text{in } C^0([0,T]),
\]
where $(A^0,\Sigma^0)$ is feasible for \eqref{eq:const-lyap}. Moreover, $A^0$ minimizes
$\trace(AA^\top)$
over all constant feasible controls. In particular,
Problem~\ref{prob:temporal-const} admits a  minimizer.
\end{thm}

\begin{proof}
Fix any constant feasible pair $(\bar A,\bar\Sigma)$ with $\bar A\in\mathcal F_{\rm const}$.
By optimality of $(A_\varepsilon,\Sigma_\varepsilon)$,
\begin{equation}\label{eq:comp-const}
J_\varepsilon(A_\varepsilon,\Sigma_\varepsilon)
\le J_\varepsilon(\bar A,\bar\Sigma)
=\varepsilon T\|\bar A\|_F^2 .
\end{equation}

\smallskip\noindent
\emph{Step 1: uniform bounds and flattening of $A_\varepsilon$.}
From \eqref{eq:comp-const},
\[
\varepsilon\int_0^T\|A_\varepsilon(t)\|_F^2\,dt
\le \varepsilon T\|\bar A\|_F^2
\quad\Rightarrow\quad
\|A_\varepsilon\|_{L^2}^2\le T\|\bar A\|_F^2 .
\]
Also,
\begin{align*}
&(1-\varepsilon)\int_0^T
\tr(\dot A_\varepsilon\Sigma_\varepsilon\dot A_\varepsilon^\top)\,dt
\le \varepsilon T\|\bar A\|_F^2
\\
&\Rightarrow\quad
\int_0^T
\tr(\dot A_\varepsilon\Sigma_\varepsilon\dot A_\varepsilon^\top)\,dt
=O(\varepsilon).
\end{align*}

To convert this into $\|\dot A_\varepsilon\|_{L^2}\to0$, we use a uniform
lower bound on $\Sigma_\varepsilon$.
Applying Theorem~\ref{eq:thm} with $\alpha=\varepsilon$ gives
\[
\Sigma_\varepsilon(t)\ge c_\varepsilon I,\qquad
c_\varepsilon
= e^{-2\sqrt{T J_\varepsilon^\star/\varepsilon}}
\,\lambda_{\min}(\Sigma_0),
\]
where $J_\varepsilon^\star=\inf J_\varepsilon$.
Since $J_\varepsilon^\star\le J_\varepsilon(\bar A,\bar\Sigma)
=\varepsilon T\|\bar A\|_F^2$, we have
\[
\frac{J_\varepsilon^\star}{\varepsilon}\le T\|\bar A\|_F^2
\quad\Rightarrow\quad
c_\varepsilon\ge
e^{-2\sqrt{T\cdot T\|\bar A\|_F^2}}\lambda_{\min}(\Sigma_0)
=:c>0,
\]
with $c$ independent of $\varepsilon$.
Therefore
\[
\int_0^T\|\dot A_\varepsilon(t)\|_F^2\,dt
\le \frac1c\int_0^T
\tr(\dot A_\varepsilon\Sigma_\varepsilon\dot A_\varepsilon^\top)\,dt
\longrightarrow 0 .
\]
Thus $\dot A_\varepsilon\to0$ in $L^2(0,T)$ and $\{A_\varepsilon\}$ is bounded in
$H^1(0,T)$. By 1D Rellich compactness, after extracting a sequence
$\varepsilon_k\downarrow0$,
\[
A_{\varepsilon_k}\to A^0\quad\text{in }C^0([0,T]),
\]
for some $A^0\in H^1$. Since $\dot A_{\varepsilon_k}\to0$ in $L^2$, the limit $A^0$
must be constant.

\smallskip\noindent
\emph{Step 2: convergence of $\Sigma_{\varepsilon_k}$ and feasibility of the limit.}
Each $\Sigma_{\varepsilon}$ solves
\[
\dot\Sigma_\varepsilon
=A_\varepsilon\Sigma_\varepsilon+\Sigma_\varepsilon A_\varepsilon^\top+BB^\top,
\qquad \Sigma_\varepsilon(0)=\Sigma_0 .
\]
Uniform spectral bounds from Theorem~\ref{eq:thm} give $\|\Sigma_\varepsilon\|_{C^0}\le C$,
and the Lyapunov equation gives equicontinuity. Hence, by Arzel\`a--Ascoli,
up to a subsequence,
\[
\Sigma_{\varepsilon_k}\to \Sigma^0 \quad\text{in }C^0([0,T]).
\]
Let $E_k:=\Sigma_{\varepsilon_k}-\Sigma^0$. Then
\[
\dot E_k
=A_{\varepsilon_k}E_k+E_kA_{\varepsilon_k}^\top
+(A_{\varepsilon_k}-A^0)\Sigma^0+\Sigma^0(A_{\varepsilon_k}-A^0)^\top.
\]
Since $A_{\varepsilon_k}\to A^0$ in $C^0$, and $\|E_k\|_{C^0}\to0$,
$\Sigma^0$ satisfies the Lyapunov equation with constant drift $A^0$.
Moreover, because each $\Sigma_{\varepsilon_k}(T)=\Sigma_{\rm fin}$ and convergence is uniform,
$\Sigma^0(T)=\Sigma_{\rm fin}$. Hence $A^0\in\mathcal F_{\rm const}$ and $(A^0,\Sigma^0)$ is feasible.

\smallskip\noindent
\emph{Step 3: selection minimality.}
Let $A\in\mathcal F_{\rm const}$ be any constant feasible control with trajectory $\Sigma$.
Since $\dot A\equiv0$,
\[
J_\varepsilon(A,\Sigma)=\varepsilon T\|A\|_F^2 .
\]
Optimality yields
\[
J_{\varepsilon_k}(A_{\varepsilon_k},\Sigma_{\varepsilon_k})
\le J_{\varepsilon_k}(A,\Sigma)
=\varepsilon_k T\|A\|_F^2 .
\]
Divide by $\varepsilon_k$ and drop the nonnegative temporal term:
\[
\int_0^T\|A_{\varepsilon_k}(t)\|_F^2\,dt
\le T\|A\|_F^2 .
\]
Passing to the limit using $A_{\varepsilon_k}\to A^0$ in $L^2$ (since $C^0$ convergence
implies $L^2$ convergence),
\[
T\|A^0\|_F^2
=\int_0^T\|A^0\|_F^2\,dt
\le T\|A\|_F^2 .
\]
Since $A$ was arbitrary in $\mathcal F_{\rm const}$, $A^0$ minimizes $\trace(AA^\top)$
over $\mathcal F_{\rm const}$. This also proves existence of a minimizer
for Problem~\ref{prob:temporal-const}.
\end{proof}

\begin{remark}
    In earlier sections we utilized the FONC to upgrade the regularity of the minimizer. Herein, however, the minimizer $A^\star$ is constant and therefore $\Sigma^\star$ is smooth. Nevertheless, the adjoint equation and stationarity conditions are given below for completeness:
    \begin{align*}
       -\dot\Lambda_t&=A^{\star\top}\Lambda_t+\Lambda_tA^\star,\\
2TA^\star &+ \int_0^T\!\big(\Lambda_t\Sigma_t^\star+\Sigma_t^\star\Lambda_t\big)\,\text{d}t=0,
    \end{align*}
    where $\Lambda:[0,T]\rightarrow \mathbb{S}^n$ is the (symmetric) adjoint variable. 
\end{remark}

\begin{remark}
    In the zero-noise limit, where $B\equiv 0$, a feasible (constant) $A$
    must satisfy $\Sigma_T=e^{AT}\Sigma_0 e^{A^\top T}$. Thereby, it must be of the form
    \begin{equation}\label{eq:orth-param}
e^{AT}=\Sigma_T^{1/2}\,R\,\Sigma_0^{-1/2},\qquad R\in O(n).
\end{equation}
Thus, Problem \ref{prob:temporal-const} can be recast as
\begin{equation}\label{eq:temporal-procrustes}
\min_{R\in O(n)}\ \Big\|\log\!\big(\Sigma_T^{1/2}R\Sigma_0^{-1/2}\big)\Big\|_F^2.
\end{equation}
This is a nonconvex ``logarithmic Procrustes'' problem over $O(n)$, and
does not generally admit a closed-form minimizer.
\end{remark}

 While the selection based on $\|A\|_F^2$ is natural from the temporal--attention viewpoint,
it does not acknowledge the covariance structure, in general. Thus, it is of interest to explore connection with information-theoretic alternatives. This is done next.

\section{Attention functional and the Fisher-Rao geodesic in the zero-noise limit}\label{sec:fisher}
In this section, we discuss a connection between the attention functional $J_\alpha$, $\alpha\in[0,1]$, defined in \eqref{eq:J}, and the closeness of its minimizer ($A^\star,\Sigma^\star$) to the canonical pair $(A^F,\Sigma^F)$ that induces the Fisher-Rao geodesic \cite{AmariNagaoka2000,Bhatia2007}
$$\Sigma_t^F=\Sigma_0^{1/2} M^{t/T}\Sigma_0^{1/2},~~t\in[0,T],$$
between $\Sigma_0$ and $\Sigma_T$ in the zero-noise limit (i.e., $B=0$), where
$M:=\Sigma_0^{-1/2}\Sigma_T\Sigma_0^{-1/2}>0$. To this end, we characterize first all Fisher geodesic-inducing \textit{constant} control matrices.
\begin{prop}[Fisher-inducing generators]
Let $T>0$, $\Sigma_0,\Sigma_T\in\mathbb S_{+}^n$.
 Then, the set of matrices $A\in\mathbb R^{n\times n}$ satisfying $\dot{\Sigma}_t^F =A\Sigma_t^F+\Sigma_t^F A^\top,t\in[0,T]$, is given by the affine space 
\begin{align*}
   \mathcal{A} :=  A^F+\Sigma_0^{1/2}L_M\Sigma_0^{-1/2},
\end{align*}
where  $A^F:=\frac{1}{2T}\Sigma_0^{1/2}(\log M)\Sigma_0^{-1/2}$ and
\begin{align*}
    L_M:=\{X\in\mathbb{A}^n~|~[X,M]=0\}.
\end{align*}
In particular, $A$ is unique (i.e., $L_M=\{0\}$) whenever $M$ has simple spectrum, and the unique choice with $A\Sigma_0=\Sigma_0A^\top$ is precisely $A=A^F$.
\end{prop}

\begin{proof}
Let $\Gamma_t= M^{t/T}$ and $C:=\log M\in \mathbb{S}^n$ the principal logarithm. Then, $\Gamma_t=e^{(t/T)C}$ and $\dot\Gamma_t=(1/T)\Gamma_t C$. 
Write $Y:=\Sigma_0^{-1/2}A\Sigma_0^{1/2}$ where $A$ satisfies $\dot{\Sigma}_t^F =A\Sigma_t^F+\Sigma_t^F A^\top,t\in[0,T]$. Then, 
\[
Y\Gamma_t+\Gamma_t Y^\top \;=\; \tfrac{1}{T}\,\Gamma_t C,\qquad \forall t\in[0,T].
\]
Left-multiplying by $\Gamma_t^{-1}$ and setting $s=t/T$ gives
\[
e^{-sC}Ye^{sC}+Y^\top=\tfrac{1}{T}C,\qquad \forall s\in[0,1].
\]
Differentiating in $s$ yields $-[C,e^{-sC}Ye^{sC}]=0$ for all $s$, hence $[C,Y]=0$. 
Decomposing $Y$ into its symmetric and skew-symmetric parts $Y=S+X$, the identity at $s=0$ gives $S=C/(2T)$, while $[C,Y]=0$ forces $[C,X]=0$, and hence $X\in L_M$. 
Conversely, if $Y=\tfrac{1}{2T}C+X$ with $X\in L_M$, then $Y\Gamma_t+\Gamma_t Y^\top=\tfrac{1}{T}\Gamma_t C$.
Returning to $A=\Sigma_0^{1/2}Y\Sigma_0^{-1/2}$ yields $\dot{\Sigma}_t^F =A\Sigma_t^F+\Sigma_t^F A^\top$.
Finally, if $M$ has a simple spectrum, then the commutant of its principal logarithm, $\mathrm{Comm}(C)$, is given by all polynomials in $C$. Hence, the only $X\in\mathbb{A}^n$ commuting with $C$ is $X=0$. This proves uniqueness (more generally, degeneracies in the multiplicities $m_j$ of eigenvalues of $M$ make $L_M\simeq \bigoplus_j \mathbb{A}^{m_j}$). This completes the proof.
\end{proof}
Using the above proposition, we show next that any minimizer of $J_\alpha$ is close to the Fisher–Rao pair $(A^{F},\Sigma^{F})$ in the sense that a \textit{Fisher geodesic-inducing cost} is kept small.
\begin{lemma}\label{lem:Fisher-vs-Attention}
Let $T>0$, $\beta\in(0,1)$, $\Sigma_0,\Sigma_T\in\mathbb S_{+}^n$, and $(A,\Sigma)\in\mathcal{H}$ satisfying $\dot\Sigma_t = A_t\Sigma_t+\Sigma_tA_t^\top $ a.e. on $[0,T]$. Define the \textit{Fisher geodesic-inducing cost:} 
\begin{align}
    \mathcal F_\beta(A,\Sigma)
&:= \beta\int_0^T\|A_t-\Sigma_tA_t^\top\Sigma_t^{-1}\|_{\Sigma_t}^2\mathrm dt\nonumber
\\
&+(1-\beta)\int_0^T\mathrm{tr}\big(\dot A_t\Sigma_t\dot A_t^\top\big)\mathrm dt,\label{eq:FisherCost}
\end{align}
where $\|X\|_{\Sigma}^2:=\mathrm{tr}(X\Sigma X^\top)$. Then, $(A^{F},\Sigma^{F})$ is the unique minimizer of \eqref{eq:FisherCost}. Moreover, if $0 < cI \leq \Sigma_t\leq CI$ for a.a. $t\in[0,T]$, then 
\begin{align}\label{eq:tight-F-by-J}
\mathcal F_\beta(A,\Sigma)\ \le\ 
K J_\alpha(A,\Sigma),
\end{align}
where $K:=\max(4\beta C^2/(\alpha c), (1-\beta)/(1-\alpha))$.
\end{lemma}

\begin{proof}
We first show that $(A^{F},\Sigma^{F})$ is the unique minimizer of \eqref{eq:FisherCost}. For the Fisher pair $(A^{F},\Sigma^F)$, $A^{F}$ is constant and $A^{F}\Sigma_t^{F}=\Sigma_t^{F}(A^{F})^\top$ holds for all $t$. Hence, $\mathcal F_\beta(A^{F},\Sigma^{F})=0$. Conversely, if $\mathcal F_\beta(A,\Sigma)=0$ then the temporal term yields $\dot A_t\equiv0$ (hence $A_t\equiv A$ is constant) and $A\Sigma_t=\Sigma_tA^\top$ for all $t$. In congruence coordinates $\Gamma_t:=\Sigma_0^{-1/2}\Sigma_t\Sigma_0^{-1/2}$ and $Y:=\Sigma_0^{-1/2}A\Sigma_0^{1/2}$ this implies \(\Gamma_t=e^{2tY}\). At $t=T$ this is \(e^{2TY}=M\). The symmetric logarithm is unique, hence \(Y=\tfrac1{2T}\log M\), i.e., $(A,\Sigma)=(A^{F},\Sigma^F)$.

Next, we prove \eqref{eq:tight-F-by-J}. To this end, let $B_t:=\Sigma_t^{-1/2}A_t\Sigma_t^{1/2}$ and  $D_t = (B_t-B_t^\top)/2$. Then
\begin{align*}
&\|A_t-\Sigma_tA_t^\top\Sigma_t^{-1}
\|_{\Sigma_t}^2= \|\Sigma_t^{1/2}(B_t-B_t^\top)\Sigma_t^{-1/2}\|^2_{\Sigma_t} \\
&=\|2\Sigma_t^{1/2}D_t\Sigma_t^{-1/2}\|_{\Sigma_t}^2=4\|D_t\|_{\Sigma_t}^2\leq 4C\|B_t\|_F^2\\
&\leq (4C^2/c)\tr(\Sigma_t^{-1}B_t^\top\Sigma_tB_t)=(4C^2/c)\tr(AA^\top).
\end{align*}
Thus, the first term in $\mathcal{F}_\beta(A,\Sigma)$ is bounded as follows
\begin{align*}
&\beta\int_0^T\|A_t-\Sigma_tA_t^\top\Sigma_t^{-1}
\|_{\Sigma_t}^2\text{d}t \leq \frac{4\beta C^2}{c\alpha}\alpha\int_0^T\|A_t\|_F^2\text{d}t.
\end{align*}
Adding the second term in $\mathcal{F}_\beta(A,\Sigma)$ to both sides of the above inequality  yields \eqref{eq:tight-F-by-J}.
\end{proof}


\section{Numerical Example} \label{sec:examples}
\begin{figure*}[t]
\centering
\begin{subcaptionblock}{\linewidth}
    \begin{subfigure}[c]{0.4\textwidth}
        \centering
        \includegraphics[width=\linewidth]{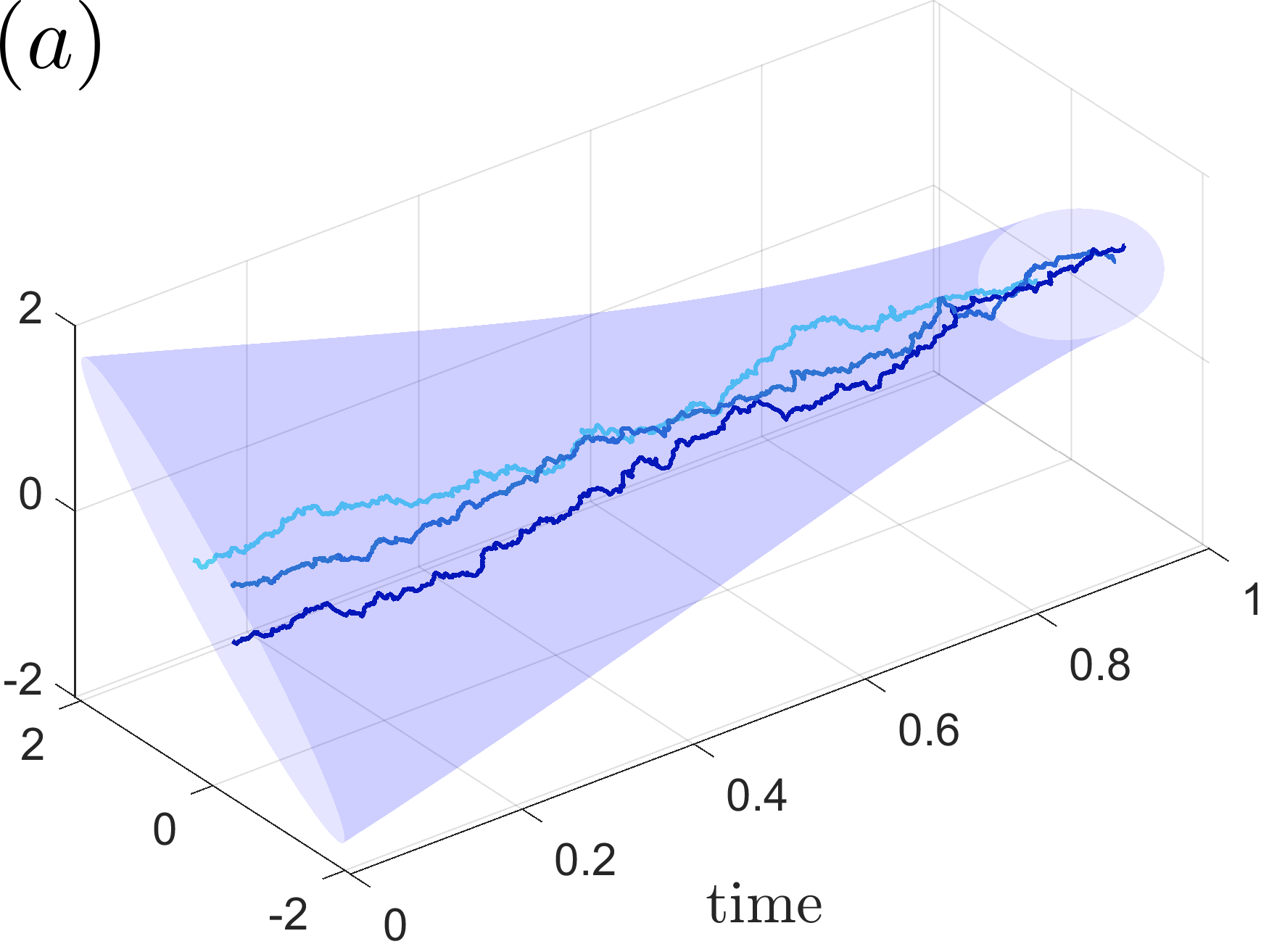}
    \end{subfigure}
    \begin{subfigure}[c]{0.3\textwidth}
        \centering
        \includegraphics[width=\linewidth]{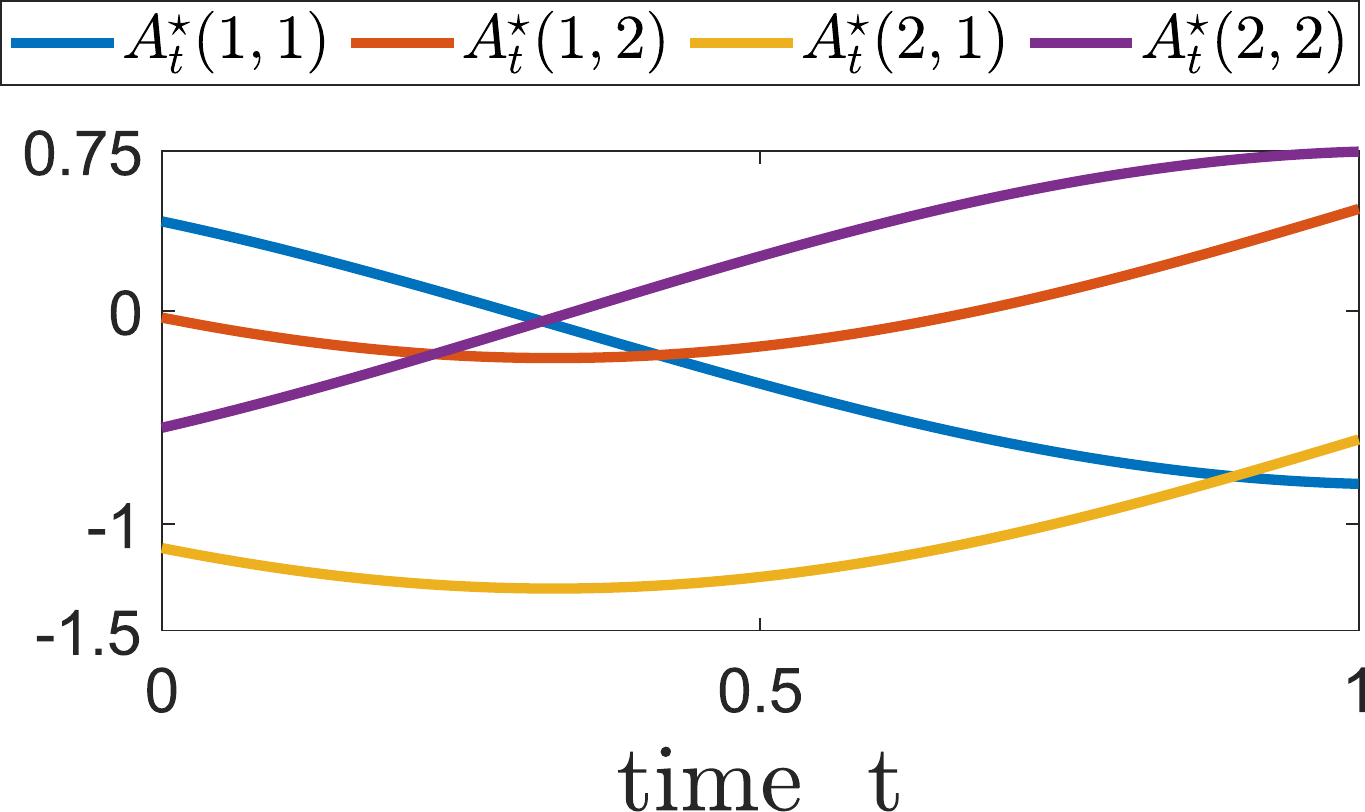}
    \end{subfigure}
     \begin{subfigure}[c]{0.3\textwidth}
        \centering
        \includegraphics[width=\linewidth]{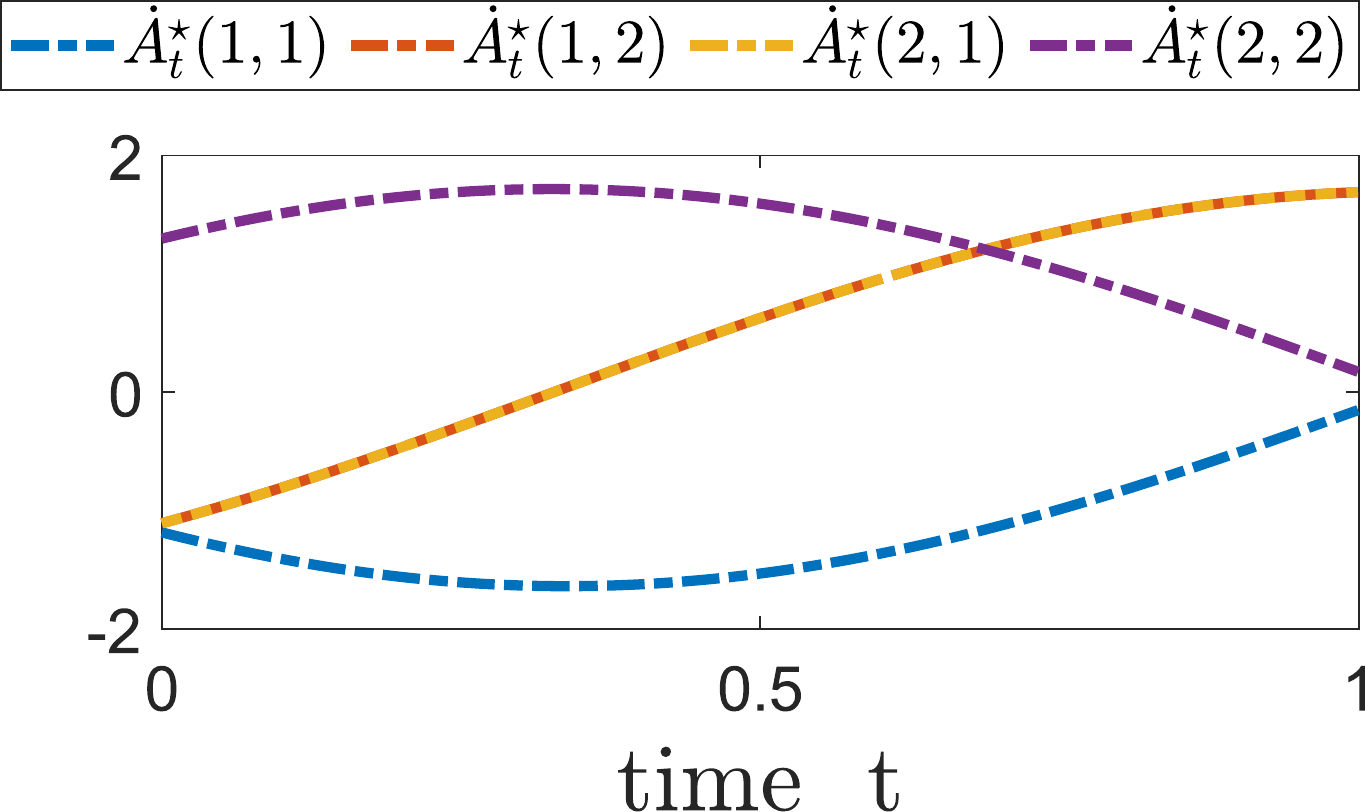}
    \end{subfigure}
    \end{subcaptionblock}
    
    \begin{subcaptionblock}{\linewidth}
    \begin{subfigure}[c]{0.4\textwidth}
    \centering
        \includegraphics[width=\linewidth]{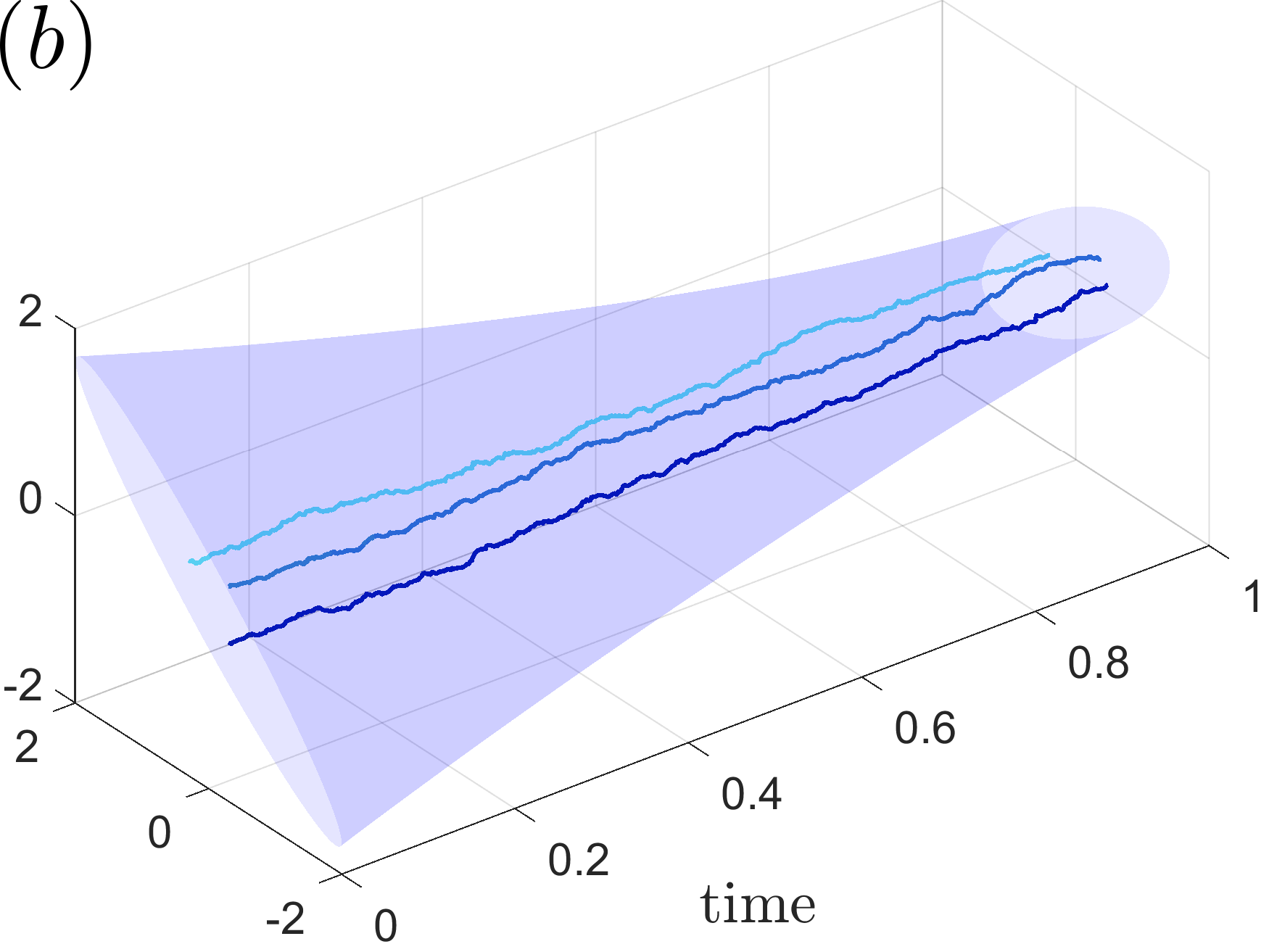}
    \end{subfigure}
    \begin{subfigure}[c]{0.3\textwidth}
    \centering
        \includegraphics[width=\linewidth]{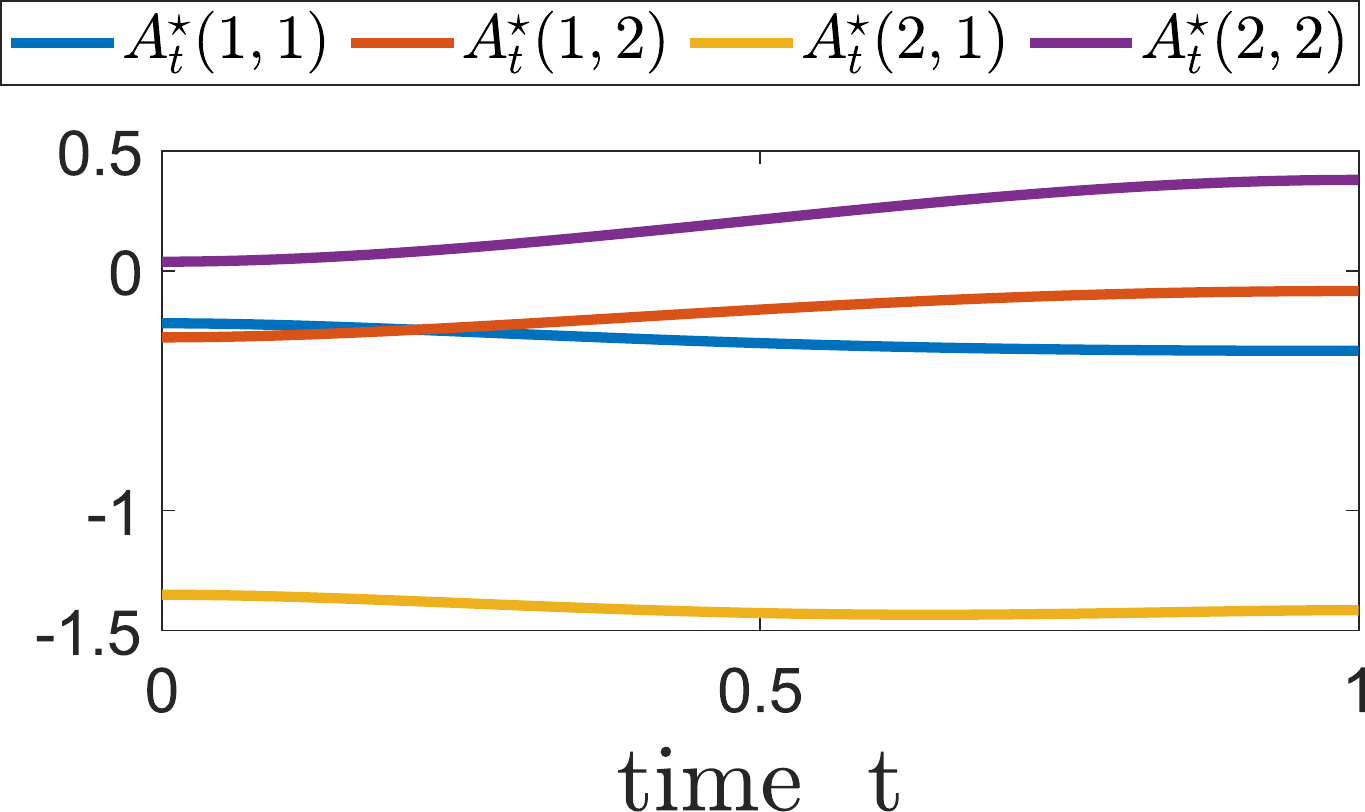}
    \end{subfigure}
     \begin{subfigure}[c]{0.3\textwidth}
     \centering
        \includegraphics[width=\linewidth]{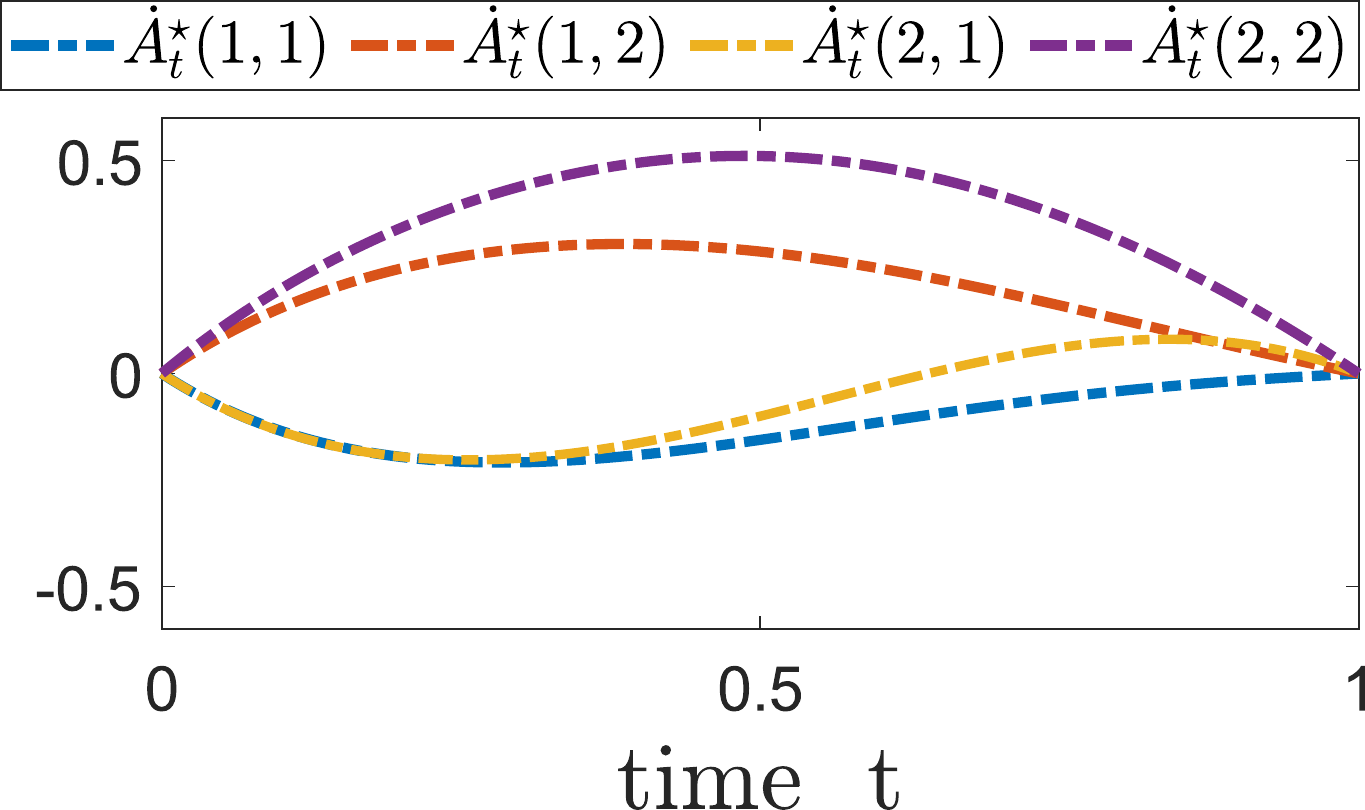}
    \end{subfigure}
    \end{subcaptionblock}
\begin{subcaptionblock}{\linewidth}
    \begin{subfigure}[c]{0.4\textwidth}
        \centering
        \includegraphics[width=\linewidth]{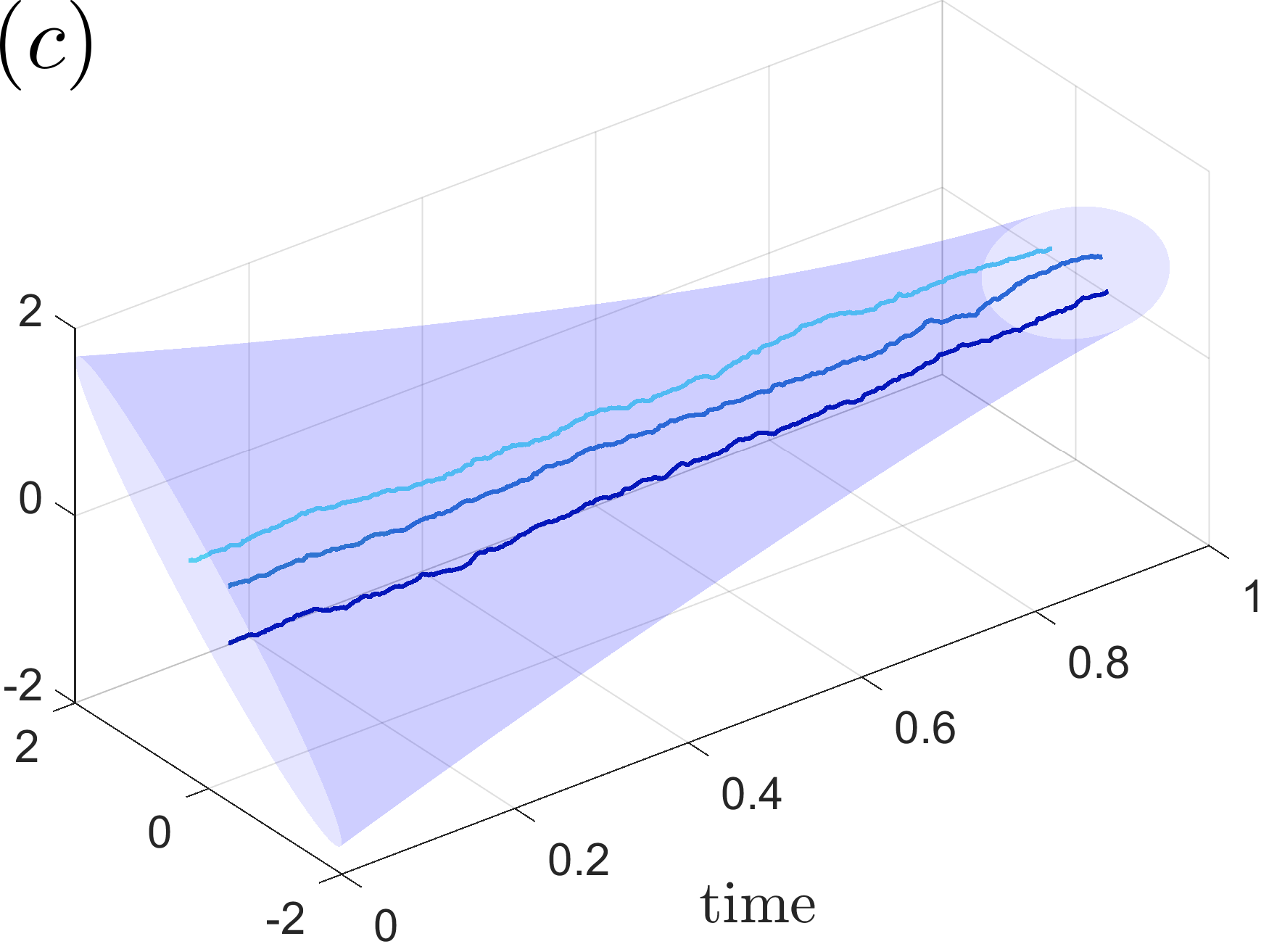}
    \end{subfigure}
    \begin{subfigure}[c]{0.3\textwidth}
        \centering
        \includegraphics[width=\linewidth]{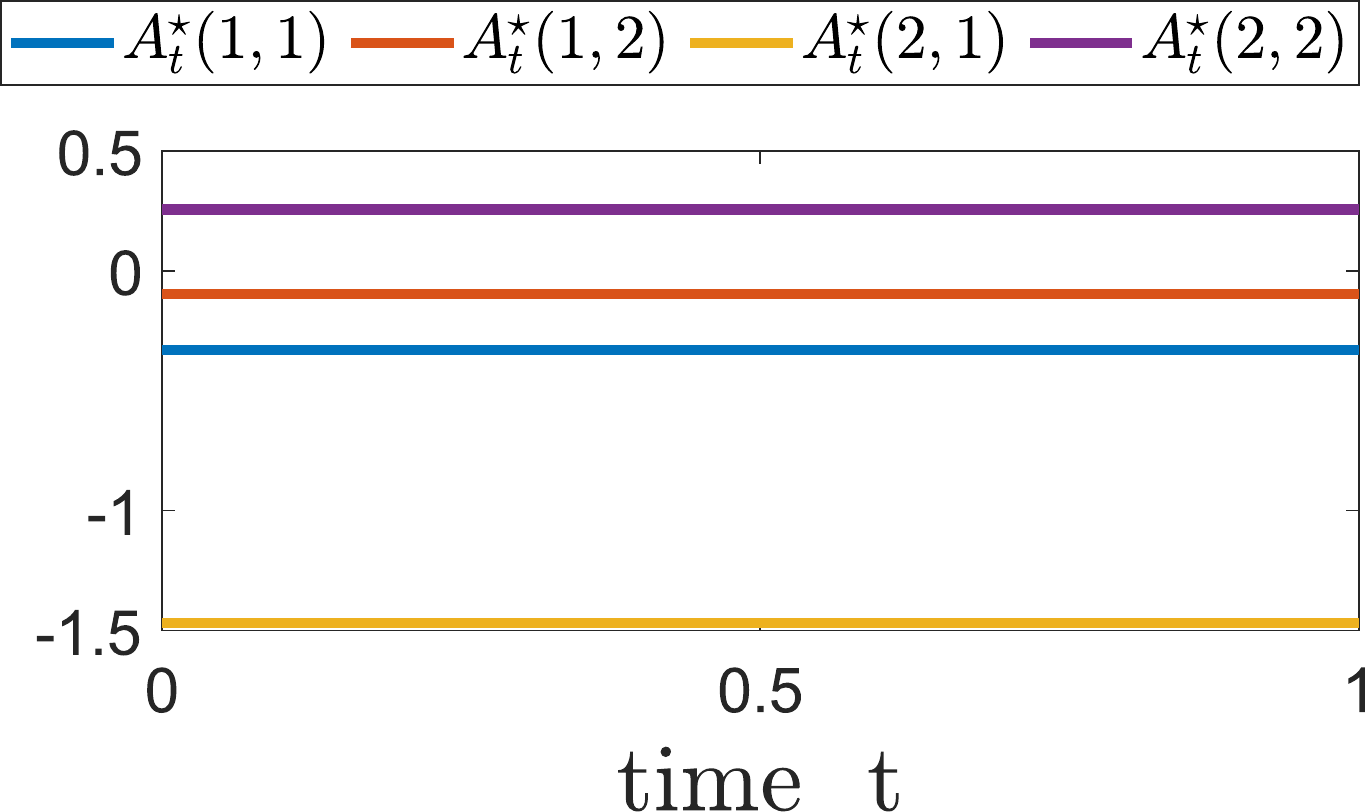}
    \end{subfigure}
     \begin{subfigure}[c]{0.3\textwidth}
        \centering
        \includegraphics[width=\linewidth]{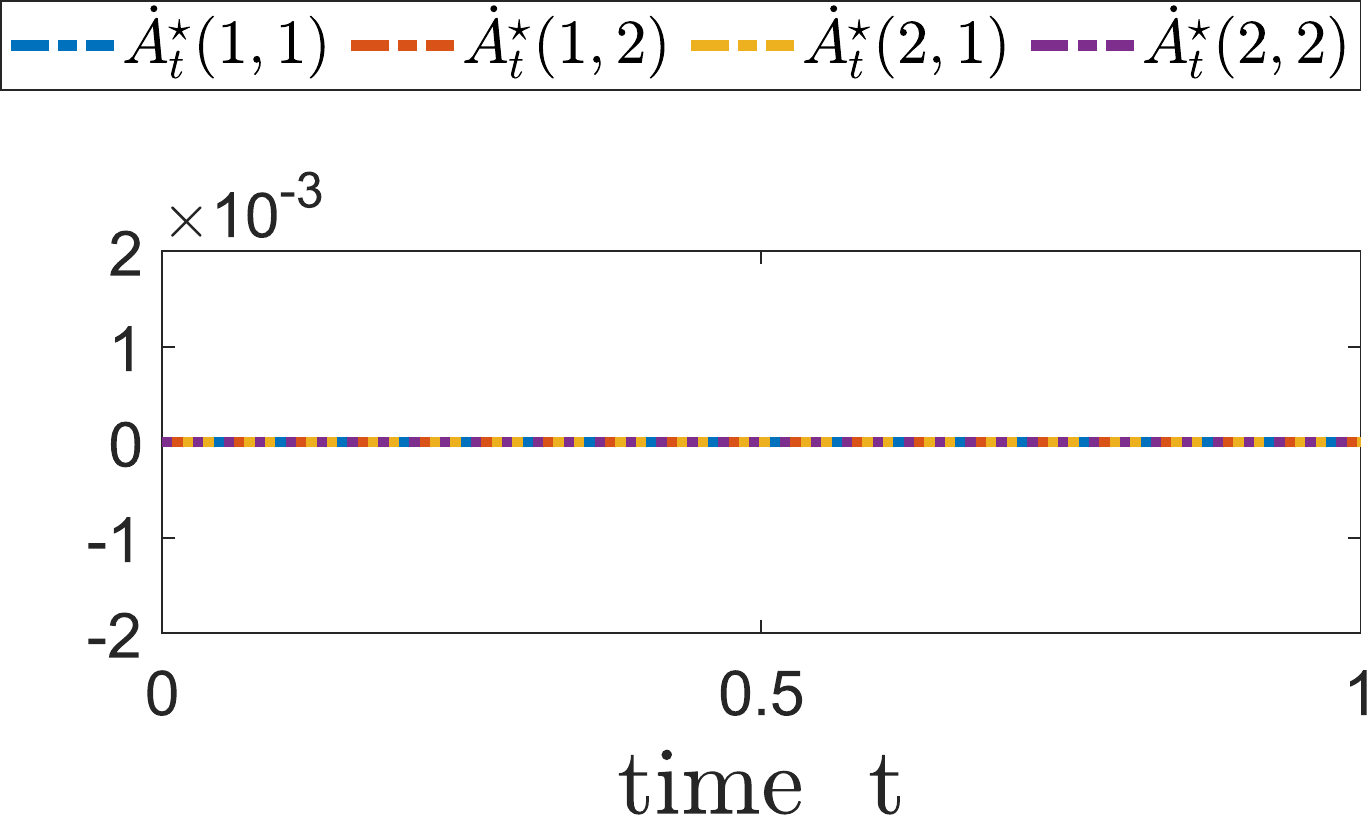}
    \end{subfigure}
    \end{subcaptionblock}
    \caption{Numerical solutions of the two-point boundary value problem \eqref{eq:neccessary-Eqs} for different cost weighing $\alpha$.}
    \end{figure*}

{\color{black}
We present an academic example for Problem~\ref{problem:main-min-attention}, where $n=m=2$ and the flow information is specified as follows:
\begin{align*}
\Sigma_{\rm init} = \begin{bmatrix}
    4 & \sqrt{11} \\
    \sqrt{11} & 3
\end{bmatrix},  \Sigma_{\rm fin} = \begin{bmatrix}
    2 & -1, \\
    -1 & 1
\end{bmatrix}, B = \frac{1}{5} I.
\end{align*}
We solve the corresponding two-point boundary value problem defined by the system \eqref{eq:neccessary-Eqs} using MATLAB's solver \texttt{bvp5c} for the case $\alpha=1/2$. The associated results are depicted in Fig.~1(b). 
The left panel illustrates the equi-probability level sets of the corresponding Gaussian distributions that connect the specified endpoints, along with sample trajectories of the associated process.  The evolutions of the corresponding control matrix and its derivative are shown in the middle and right panels, respectively. The slight variation in the control signal over time is indeed due to the minimization of the attention functional. 

The solution at $\alpha=1/2$ is used to initialize the numerical solver to find nearest solutions at the cases $\alpha=1$ and $\alpha=0$. The corresponding results are presented in Fig.~1(a) and Fig.1~(c), respectively.  The substantial variation in the control over time for the case $\alpha=1$ is a result of excluding the temporal attention from the cost functional \eqref{eq:J}. Conversely, the constant control signal in the case $\alpha=0$ is a result of omitting the spatial attention from the cost. 

}

\section{Concluding remarks}
\label{sec:conclusions}
The purpose of this work has been to revisit the concept of `attention.' This concept was introduced by Roger Brockett in order to quantify what one would colloquially describe as the attention required to adjust the applied control action based on state and timing information.
Our intention has been to formalize and study Brockett's concept for linear dynamical systems and Gaussian statistics. This endeavor necessitated introducing an `averaged version of attention,'' as expressed in \eqref{eq:J}, weighed over regions of the state space in accordance with a relevant probability law.

On the surface, the attention functional appears similar to expressions in linear Gaussian-Markov system theory. However, this new paradigm does not share the enabling convexity of the linear-quadratic theory, and as a consequence, establishing existence and regularity of solutions proved technically demanding. It is carried out folloing the general lines of the direct method in the calculus of variations.

In particular, we analyzed separately the two special case of the attention functional--the \emph{spatial} and the \emph{temporal}, that provide natural extreme choice of attention in these two extreme instances when on is interested in the purely spatial, or temporal, variability of the control protocol. We established existence and regularity of minimizers in these cases as well, and in a certain setting, we compared with an alternative information-theoretic functional that gives rise to Fisher–Rao geodesics. Our development complements and parallels the theory of Gaussian steering/Schr\"odinger bridge for linear systems \cite{ChenGeorgiouPavon2016_TAC1,ChenGeorgiouPavon2016_TAC2,Leonard2014SchrodingerSurvey} and speaks to resource–aware control viewpoints (periodic/event–triggered actuation and information constraints) \cite{HeemelsDonkersTeel2013}.

We expect that our treatment of attention–penalized Gaussian steering will provide the basis for theoretical advances as well as expand the range of applications by helping to balance sensing/actuation effort and cost.
Further, it is our hope that future work will illuminate apparent links between attention and sheer in fluid flow \cite{Arnold1966,EbinMarsden1970,HolmMarsdenRatiu1998} possibly leading to useful variational formalisms \cite{MerigotMirebeau2016,BenamouCarlierNenna2019,ArnaudonCruzeiroLeonardZambrini2020,BaradatMonsaingeon2020} that may drive the latter.

\bibliographystyle{IEEEtran}
\bibliography{main}

@book{wiener2019cybernetics,
  title={Cybernetics or Control and Communication in the Animal and the Machine},
  author={Wiener, Norbert},
  year={2019},
  publisher={MIT press}
}

@article{Arnold1966,
  author  = {Arnold, V. I.},
  title   = {Sur la g{\'e}om{\'e}trie diff{\'e}rentielle des groupes de Lie de dimension infinie et ses applications {\'a} l'hydrodynamique des fluides parfaits},
  journal = {Annales de l'Institut Fourier},
  volume  = {16},
  number  = {1},
  pages   = {319--361},
  year    = {1966},
  doi     = {10.5802/aif.233}
}

@article{EbinMarsden1970,
  author  = {Ebin, D. G. and Marsden, J. E.},
  title   = {Groups of diffeomorphisms and the motion of an incompressible fluid},
  journal = {Annals of Mathematics},
  volume  = {92},
  number  = {1},
  pages   = {102--163},
  year    = {1970},
  doi     = {10.2307/1970699}
}

@article{HolmMarsdenRatiu1998,
  author  = {Holm, D. D. and Marsden, J. E. and Ratiu, T. S.},
  title   = {The {Euler--Poincar{\'e}} equations and semidirect products with applications to continuum theories},
  journal = {Advances in Mathematics},
  year    = {1998}
}

@book{Rudin1991FA,
  author    = {Walter Rudin},
  title     = {Functional Analysis},
  edition   = {2},
  year      = {1991},
  publisher = {McGraw-Hill},
  address   = {New York},
  series    = {International Series in Pure and Applied Mathematics},
  isbn      = {0-07-054236-8}
}

@book{bhatia2013matrix,
  title={Matrix analysis},
  author={Bhatia, Rajendra},
  volume={169},
  year={2013},
  publisher={Springer Science \& Business Media}
}

@book{GelfandVilenkin1964,
  title     = {Generalized Functions, Volume 4: Applications of Harmonic Analysis},
  author    = {I. M. Gel'fand and N. Ya. Vilenkin},
  year      = {1964},
  publisher = {Academic Press},
  address   = {New York},
  note      = {Classical source on rigged Hilbert spaces (Gel'fand triples) in the distribution-theoretic framework}
}

@book{BonnansShapiro2000,
  author    = {J. Fr\'ed\'eric Bonnans and Alexander Shapiro},
  title     = {Perturbation Analysis of Optimization Problems},
  series    = {Springer Series in Operations Research},
  publisher = {Springer},
  address   = {New York},
  year      = {2000},
  isbn      = {978-0-387-98773-6},
  note      = {See the chapter on first-order optimality conditions for Banach-space KKT with surjective constraint derivative}
}

@book{AmariNagaoka2000,
  author    = {{Shun-ichi} Amari and Hiroshi Nagaoka},
  title     = {Methods of Information Geometry},
  series    = {Translations of Mathematical Monographs},
  volume    = {191},
  publisher = {American Mathematical Society},
  year      = {2000},
  address   = {Providence, RI}
}

@book{Bhatia2007,
  author    = {Rajendra Bhatia},
  title     = {Positive Definite Matrices},
  publisher = {Princeton University Press},
  year      = {2007},
  address   = {Princeton, NJ}
}

@book{Astrom,
title = "Introduction to stochastic control theory",
author = "{\AA}str{\"o}m, \{Karl Johan\}",
year = "1970",
language = "English",
isbn = "0-12-065650-7",
volume = "70",
series = "Mathematics in science and engineering",
publisher = "Academic Press",
address = "United States",
}

@article{ChenGeorgiouPavon2016_TAC1,
  author  = {Yongxin Chen and Tryphon T. Georgiou and Michele Pavon},
  title   = {Optimal Steering of a Linear Stochastic System to a Final Probability Distribution, Part {I}},
  journal = {IEEE Transactions on Automatic Control},
  year    = {2016},
  volume  = {61},
  number  = {5},
  pages   = {1158--1169}
}

@article{ChenGeorgiouPavon2016_TAC2,
  author  = {Yongxin Chen and Tryphon T. Georgiou and Michele Pavon},
  title   = {Optimal Steering of a Linear Stochastic System to a Final Probability Distribution, Part {II}},
  journal = {IEEE Transactions on Automatic Control},
  year    = {2016},
  volume  = {61},
  number  = {5},
  pages   = {1170--1185}
}

@book{Troeltzsch2010,
  author    = {Fredi Tr\"oltzsch},
  title     = {Optimal Control of Partial Differential Equations: Theory, Methods and Applications},
  series    = {Graduate Studies in Mathematics},
  volume    = {112},
  publisher = {American Mathematical Society},
  address   = {Providence, RI},
  year      = {2010},
  isbn      = {978-0-8218-4904-0},
  note      = {Early chapters present the Banach-space Lagrange multiplier theorem for equality constraints (C\^1 mapping, surjective derivative)}
}

@book{evans2022partial,
  title={Partial differential equations},
  author={Evans, Lawrence C},
  volume={19},
  year={2022},
  publisher={American mathematical society}
}

@inproceedings{Brockett1997MinAttn,
  author    = {Roger W. Brockett},
  title     = {Minimum Attention Control},
  booktitle = {Proceedings of the 36th IEEE Conference on Decision and Control},
  year      = {1997},
  pages     = {2628--2632},
  address   = {San Diego, CA, USA}
}

@inproceedings{Brockett2003MinimizingAttention,
  author    = {Roger W. Brockett},
  title     = {Minimizing Attention in a Motion Control Context},
  booktitle = {Proceedings of the 42nd IEEE Conference on Decision and Control},
  year      = {2003},
  pages     = {3349--3352},
  address   = {Maui, HI, USA},
  doi       = {10.1109/CDC.2003.1271661}
}

@incollection{Brockett2012Liouville,
  author    = {Roger W. Brockett},
  title     = {Notes on the Control of the {L}iouville Equation},
  booktitle = {Control of Partial Differential Equations},
  editor    = {Patrizia Cannarsa and Jean{-}Michel Coron},
  series    = {Lecture Notes in Mathematics},
  volume    = {2048},
  publisher = {Springer},
  address   = {Berlin, Heidelberg},
  year      = {2012},
  pages     = {101--129},
  doi       = {10.1007/978-3-642-27893-8_2}
}

@article{Donkers2014MinAttnLinear,
  author  = {M. C. F. Donkers and P. Tabuada and W. P. M. H. Heemels},
  title   = {Minimum Attention Control for Linear Systems: {A} Linear Programming Approach},
  journal = {Discrete Event Dynamic Systems},
  volume  = {24},
  number  = {2},
  pages   = {199--218},
  year    = {2014},
  doi     = {10.1007/s10626-012-0155-x}
}

@inproceedings{Anta2010MinAttnAnytime,
  author    = {Adolfo Anta and Paulo Tabuada},
  title     = {On the Minimum Attention and Anytime Attention Problems for Nonlinear Systems},
  booktitle = {Proceedings of the 49th IEEE Conference on Decision and Control},
  year      = {2010},
  pages     = {3234--3239},
  address   = {Atlanta, GA, USA}
}

@article{Jang2015BallCatching,
  author  = {Cheongjae Jang and Jee{-}eun Lee and Sohee Lee and Frank C. Park},
  title   = {A Minimum Attention Control Law for Ball Catching},
  journal = {Bioinspiration {\&} Biomimetics},
  volume  = {10},
  number  = {5},
  pages   = {055008},
  year    = {2015},
  doi     = {10.1088/1748-3190/10/5/055008}
}

@article{Dirr2016EnsembleMeanVariance,
  author  = {G. Dirr and U. Helmke and M. Sch{\"o}nlein},
  title   = {Controlling Mean and Variance in Ensembles of Linear Systems},
  journal = {IFAC-PapersOnLine},
  volume  = {49},
  number  = {18},
  pages   = {1018--1023},
  year    = {2016}
}

@inproceedings{NagaharaNesic2020CDC,
  author    = {Masaaki Nagahara and Dragan Ne{\v s}i{\'c}},
  title     = {An Approach to Minimum Attention Control by Sparse Optimization},
  booktitle = {2020 59th IEEE Conference on Decision and Control (CDC)},
  year      = {2020},
  pages     = {4205--4210},
  doi       = {10.1109/CDC42340.2020.9303783}
}

@article{eldesoukey2025collective,
  title={Collective steering: Tracer-informed dynamics},
  author={Eldesoukey, Asmaa and Abdelgalil, Mahmoud and Georgiou, Tryphon T},
  journal={arXiv preprint arXiv:2505.01975},
  year={2025}
}

@article{nowzari2016distributed,
  title={Distributed event-triggered coordination for average consensus on weight-balanced digraphs},
  author={Nowzari, Cameron and Cort{\'e}s, Jorge},
  journal={Automatica},
  volume={68},
  pages={237--244},
  year={2016},
  publisher={Elsevier}
}

@inproceedings{heemels2012introduction,
  title={An introduction to event-triggered and self-triggered control},
author={Heemels, Wilhelmus PMH and Johansson, Karl Henrik and Tabuada, Paulo},
  booktitle={51st {IEEE} {C}onference on {D}ecision and {C}ontrol},
  pages={3270--3285},
  year={2012},
  organization={IEEE}
}

@article{chen2021stochastic,
  title={Stochastic control liaisons: {R}ichard {S}inkhorn meets {G}aspard {M}onge on a {S}chr\"odinger bridge},
  author={Chen, Yongxin and Georgiou, Tryphon T and Pavon, Michele},
  journal={SIAM Review},
  volume={63},
  number={2},
  pages={249--313},
  year={2021},
  publisher={SIAM}
}

@book{Villani2009OptimalTransport,
  author    = {C{\'e}dric Villani},
  title     = {Optimal Transport: Old and New},
  publisher = {Springer},
  year      = {2009}
}

@article{Leonard2014SchrodingerSurvey,
  author    = {Christian L{\'e}onard},
  title     = {A survey of the {S}chr{\"o}dinger problem and some of its connections with optimal transport},
  journal   = {Discrete and Continuous Dynamical Systems~A},
  year      = {2014},
  volume    = {34},
  number    = {4},
  pages     = {1533--1574}
}

@article{HeemelsDonkersTeel2013,
  author    = {W. P. M. H. Heemels and M. C. F. Donkers and A. R. Teel},
  title     = {Periodic event-triggered control for linear systems},
  journal   = {IEEE Transactions on Automatic Control},
  year      = {2013},
  volume    = {58},
  number    = {4},
  pages     = {847--861}
}

@article{BenamouCarlierNenna2019,
  author  = {Benamou, Jean-David and Carlier, Guillaume and Nenna, Luca},
  title   = {Generalized incompressible flows, multi-marginal transport and Sinkhorn algorithm},
  journal = {Numerische Mathematik},
  volume  = {142},
  number  = {1},
  pages   = {33--54},
  year    = {2019},
  doi     = {10.1007/s00211-018-0995-x}
}

@article{ArnaudonCruzeiroLeonardZambrini2020,
  author  = {Arnaudon, Marc and Cruzeiro, Ana Bela and L{\'e}onard, Christian and Zambrini, Jean-Claude},
  title   = {An entropic interpolation problem for incompressible viscous fluids},
  journal = {Annales de l'Institut Henri Poincar{\'e} (B) Probabilit{\'e}s et Statistiques},
  volume  = {56},
  number  = {3},
  pages   = {2211--2235},
  year    = {2020},
  doi     = {10.1214/19-AIHP1036}
}

@article{BaradatMonsaingeon2020,
  author  = {Baradat, Aymeric and Monsaingeon, L{\'e}onard},
  title   = {Small Noise Limit and Convexity for Generalized Incompressible Flows, Schr{\"o}dinger Problems, and Optimal Transport},
  journal = {Archive for Rational Mechanics and Analysis},
  volume  = {235},
  number  = {2},
  pages   = {1357--1403},
  year    = {2020},
  doi     = {10.1007/s00205-019-01446-w}
}

@article{MerigotMirebeau2016,
  author  = {M{\'e}rigot, Quentin and Mirebeau, Jean-Marie},
  title   = {Minimal geodesics along volume-preserving maps, through semidiscrete optimal transport},
  journal = {SIAM Journal on Numerical Analysis},
  volume  = {54},
  number  = {6},
  pages   = {3465--3492},
  year    = {2016},
  doi     = {10.1137/15M1017235}
}

\end{document}